\documentclass[twoside]{article} \usepackage[accepted]{aistats2017}

\usepackage{times}
\usepackage{graphicx} %
\usepackage{subfigure} 

\usepackage{algorithm}
\usepackage{algorithmic}
\usepackage{natbib} 

\usepackage[utf8]{inputenc} %
\usepackage[T1]{fontenc}    %
\usepackage{hyperref}       %
\usepackage{url}            %
\usepackage{booktabs}       %
\usepackage{amsfonts}       %
\usepackage{nicefrac}       %
\usepackage{microtype}      %
\usepackage{bm,color}
\usepackage{amsmath, amssymb,amsthm}
\usepackage{wrapfig}

\providecommand{\abs}[1]{\left\lvert#1\right\rvert}
\providecommand{\norm}[1]{\left\lVert#1\right\rVert}
\providecommand{\frobnorm}[1]{\left\lVert#1\right\rVert_{\textsf{Fro}}}

\newcommand{\ball}{\mathcal{B}}

\newcommand{\domain}{\mathcal{C}}

\newcommand{\xv}{\bm{x}}
\newcommand{\yv}{\bm{y}}
\newcommand{\zv}{\bm{z}}
\newcommand{\sv}{\bm{s}}
\newcommand{\av}{\bm{a}}
\newcommand{\bv}{\bm{b}}
\newcommand{\cv}{\bm{c}}
\newcommand{\ev}{\bm{e}}

\newcommand{\uv}{\bm{u}}
\newcommand{\vv}{\bm{v}}
\newcommand{\wv}{\bm{w}}
\newcommand{\mv}{\bm{m}}
\newcommand{\GW}{G_\setC}%
\newcommand{\G}{G}

\newcommand{\id}{\bm{\iota}}
\newcommand{\alphav}{\bm{\alpha}}
\newcommand{\betav}{\bm{\beta}}

\newcommand{\epsilonv}{\bm{\epsilon}}

\newcommand{\one}{\mathbf{1}} %
\newcommand{\Simplex}{\triangle}

\newcommand{\R}{\mathbb{R}}  

\newcommand{\calP}{\mathcal{O}_{\hspace{-1pt}B}} %
\newcommand{\calD}{\mathcal{O}_{\hspace{-1pt}A}}

\newcommand{\setC}{\mathcal{C}}

\DeclareMathOperator{\dom}{dom}         %
\DeclareMathOperator{\intr}{int}         %
\DeclareMathOperator*{\argmin}{arg\,min}

\newtheorem*{rep@theorem}{\rep@title}
\newcommand{\newreptheorem}[2]{%
\newenvironment{rep#1}[1]{%
 \def\rep@title{#2 \ref{##1}}%
 \begin{rep@theorem}}%
 {\end{rep@theorem}}}
\newreptheorem{lemma}{Lemma'}
\newreptheorem{definition}{Definition'}
\newreptheorem{proposition}{Proposition'}
\newreptheorem{theorem}{Theorem'}

\theoremstyle{plain}
\newtheorem{theorem}{Theorem}
\newtheorem{lemma}[theorem]{Lemma}

\newtheorem{corollary}[theorem]{Corollary}
\theoremstyle{definition}

\begin{document}

\twocolumn[

\aistatstitle{Screening Rules for Convex Problems}

\aistatsauthor{
Anant Raj \And 
Jakob Olbrich \And 
Bernd G\"artner \And
Bernhard Sch\"olkopf \And
Martin Jaggi
}

\aistatsaddress{
MPI \And 
ETH \And 
ETH \And
MPI \And
EPFL
} ]

\begin{abstract}
We propose a new framework for deriving screening rules for convex optimization problems.
Our approach covers a large class of constrained and penalized optimization formulations, and works in two steps. First, given any approximate point, the structure of the objective function and the duality gap is used to gather information on the optimal solution. In the second step, this information is used to produce screening rules, i.e. safely identifying unimportant weight variables of the optimal solution.
Our general framework leads to a large variety of useful existing as well as new screening rules for many applications. For example, we provide new screening rules for general simplex and $L_1$-constrained problems, Elastic Net, squared-loss Support Vector Machines, minimum enclosing ball, as well as structured norm regularized problems, such as group lasso.
\end{abstract}

\section{Introduction}\vspace{-1mm}
\label{sec:intro}
\footnote{
Parts of this work have appeared in the Master's Thesis \citep{Olbrich:2015kq}.
}%
Optimization techniques for high-dimensional problems have become the work-horses for most data-analysis and machine-learning methods. With the rapid increase of available data, major challenges occur as the number of optimization variables (weights) grows beyond capacity of current systems.

The idea of screening refers to eliminating optimization variables that are guaranteed to \emph{not} contribute to any optimal solution, and can therefore safely be removed from problem.
Such screening techniques have received increased interest in several machine learning related applications in recent years, and have been shown to lead to very significant computational efficiency improvements in various cases, in particular for many types of sparse methods.
Screening techniques can be used either as a pre-processing before passing the problem to the optimizer, or also interactively during any iterative solver (called dynamic screening), to gradually reduce the problem complexity during optimization.

While existing screening methods were mainly relying on geometric and problem-specific properties, we in this paper take a different approach. We propose a new framework allowing screening on general convex optimization problems, using simple tools from convex duality instead of any geometric arguments. Our framework applies to a very large class of optimization problems both for constrained as well as penalized problems, including most machine learning methods of interest.

Our main contributions in this paper are summarized as follows:\vspace{-1mm}
\begin{enumerate}
\item We propose a new framework for screening for a more general class of optimization problem with a simple primal-dual structure.\vspace{-1pt}
\item The framework leads to a large set of new screening rules for machine learning problems that could not be screened before. Furthermore, it also recovers many existing screening rules as special cases.\vspace{-1pt}
\item We are able to express all screening rules using general optimization complexity notions  such as smoothness or strong convexity, getting rid of problem-specific geometric properties.\vspace{-1pt}
\item Our proposed rules are dynamic (allowing any existing algorithm to be additionally equipped with screening) %
and safe (guaranteed to only eliminate truly unimportant variables).\vspace{-1pt}
\end{enumerate}

\paragraph{Related Work.}

The concept of screening in the sense of eliminating non-influential data points to reduce the problem size has originated relatively independently in at least two communities.
Coming from computational geometry, \cite{Ahipasaoglu:2008il} has proposed a screening technique for the minimum enclosing ball problem for a given set of data points. Here screening can be interpreted as simply removing points which are guaranteed to lie in the strict interior of the final ball.
Later \cite{Kallberg:2014kd} improve the threshold for this rule in the minimum enclosing ball setting. 

Independently, the breakthrough work of \cite{Ghaoui:2010tu} gave the first screening rules for the important case of sparse regression, as given in the Lasso. Since then, there have been many extensions and alterations of the general concept. While \cite{Ghaoui:2010tu} exploits geometric quantities to bound the the Lasso dual solution within a compact region, we recommend the survey paper by \cite{Xiang:2014vi} for an overview of geometric methods for Lasso screening. 
Sphere-region based methods differ from dome-shaped regions as used in \cite{Ghaoui:2010tu} in choosing different centers and radii to bound the dual optimal point. 
Apart from being geometry specific, most existing approaches such as \citep{Wang:2013uq,Wang:2014un,Liu:2014uz,Ghaoui:2010tu,Ogawa:2013ul} are not agnostic to the regularization parameter used, but instead are restricted to perform screening along the entire regularization path (as the regularization parameter changes). This is known as sequential screening, and restricts its usability to optimization algorithms obtaining paths.
In contrast, our proposed framework here allows any internal optimization algorithms to be equipped with screening.

Despite the importance of constrained problems in many applications, much less is known about screening for constrained optimization, in contrast to the case of penalized optimization problems. 
For the dual of the hinge loss SVM, which is a box-constrained optimization problem, \cite{Ogawa:2014vw} proposed a geometric screening rule based on the intersection region of two spheres, in the sequential setting of varying regularizer. More recently,~\cite{Zimmert:2015ue} provided new screening rules for that case in the dynamic setting using a method similar to our approach. However their method is restricted to the SVM case. 

As a first step to allow screening for more general optimization objectives, \cite{Ndiaye:2015wj} gives duality gap based screening rules for multi-task and multi-class problems (in the penalized setting) under for a wider class of objectives~$f$. Nevertheless, their approach is restricted to assume separability of $f$ over the group structure, which limits the screening rules, in the sense of not covering standard group lasso for example. Also in \citep{Shibagaki:2016ts}, authors assume the similar problem formulation as in \citep{Ndiaye:2015wj} but a bit more general. The focus in \cite{Shibagaki:2016ts} is on screening rules for SVM problems rather than general framework.
They derive the screening rules for SVM by considering standard hinge and $\epsilon$-insensitive loss with regularization formulation which is close to the empirical risk minimization framework here, but has more limited applications in terms of generalization of the screening rules. We here provide screening rules for a more general framework of box constrained optimization, while hinge-loss SVM happens to be a special case of this.   %
Our proposed approach can be shown to recover many of the other existing rules including e.g. \citep{Ndiaye:2015wj} and \citep{Zimmert:2015ue}, but significantly generalizing the method to general objectives and constraints as well as regularizers.

The rest of the paper is organized as follows: In Section~\ref{sec:primaldual}, we discuss our framework for screening. Section~\ref{sec:duality_gap_certificates} is devoted to deriving the information about optimal points in terms of gap functions. Sections \ref{sec:screening_constrained} and \ref{sec:screening_penalized} utilizes the framework and tools derived in previous sections to provide screening rules for constrained and penalized case respectively. In the end, we provide a small illustrative experiment for screening on simplex and $L_1$-constrained and also discuss that which of the existing results can be recovered using our algorithm in Section \ref{sec:exp}.
\vspace{-2mm}

\section{Setup and Primal-Dual Structure}
\label{sec:primaldual}
In this paper, we consider optimization problems of the following primal-dual structure. As we will see, the relationship between primal and dual objectives has many benefits, including computation of the duality gap, which allows us to have a certificate for approximation quality. 

A very wide range of machine learning optimization problems can be formulated as \eqref{eq:primal} and \eqref{eq:dual}, which are dual to each other:\vspace{-1mm}
\begin{align}
    \label{eq:primal}\tag{A}\ \ \ \ %
    \min_{\xv \in \R^{n}} \quad& \Big[ \ \ 
    \calD(\xv) \,:= \ \ f(A\xv )
    \ +\ g(\xv) &\Big]\ 
\\
    \label{eq:dual}\tag{B}\ \ \ \ %
    \min_{\wv \in \R^{d}} \quad& \Big[ \ \ 
    \calP(\wv) :=  \ \ f^*(\wv )
    \ +\ g^*(-A^\top\wv) \!\!\!&\Big]\ 
\end{align}
The two problems are associated to a given data matrix $A\in\R^{d\times n}$, and the functions $f : \R^d \rightarrow \R$ and $g : \R^n \rightarrow \R$ are allowed to be arbitrary closed convex functions.
The functions $f^*,g^*$ in formulation~\eqref{eq:dual} are defined as the \textit{convex conjugates} of their corresponding counterparts $f,g$ in~\eqref{eq:primal}.
Here $\xv \in \R^n$ and $\wv \in \R^d$ are the respective variable vectors.
For a given function $h: \R^d\rightarrow \R$, its conjugate is defined as
\[
h^*(\vv) := \max_{\uv\in\R^d} \vv^\top \uv - h(\uv) \, .
\]
The association of problems \eqref{eq:primal} and \eqref{eq:dual} is a special case of Fenchel Duality. More precisely, the relationship is called \emph{Fenchel-Rockafellar Duality} when incorporating the linear map $A$ as in our case, see e.g. \cite[Theorem 4.4.2]{Borwein:2005ge} or \cite[Proposition 15.18]{Bauschke:2011ik}, see the Appendix~\ref{app:primaldual} for a self-contained derivation. 
The two main powerful features of this general duality structure are first that it includes many more machine learning methods than more traditional duality notions, and secondly that the two problems are fully symmetric, when changing respective roles of~$f$ and~$g$. %
In typical machine learning problems, the two parts typically play the roles of a data-fit (or loss) term as well as a regularization term. As we will see later, the two roles can be swapped, depending on the application.

\paragraph{Optimality Conditions.}
The first-order optimality conditions for our pair of vectors $\wv\in \R^d, \xv \in \R^n$ in problems~\eqref{eq:primal} and~\eqref{eq:dual} are given as
\vspace{-4mm}

\begin{minipage}{0.44\columnwidth}
  \begin{subequations}
  \begin{align}
 \wv \in&\ \partial f(A\xv) \label{eq:opt_f} \ ,\\ 
A\xv \in&\ \partial f^*(\wv) \label{eq:opt_fstar} \ ,
\end{align}
  \end{subequations}
\end{minipage}%
\hfill
\begin{minipage}{0.54\columnwidth}
  \begin{subequations}
\begin{align}
 -A^\top \wv \in&\ \partial g(\xv) \label{eq:opt_g} \ ,\\ 
 \xv \in&\ \partial g^*(-A^\top\wv) \!\!\!\!\!\!&\label{eq:opt_gstar}
\end{align}
  \end{subequations}
\end{minipage}

see e.g. \cite[Proposition 19.18]{Bauschke:2011ik}. The stated optimality conditions are equivalent to $\xv,\wv$ being a saddle-point of the Lagrangian, which is given as $\mathcal{L}(\xv,\wv) =f^*(\wv) - \langle A\xv,\wv\rangle - g(\xv) $ if $\xv\in \dom(g)$ and $\wv \in \dom (f^*)$, see Appendix~\ref{app:primaldual} for details.

\paragraph{The Constrained Case.}
Any constrained convex optimization problem of the form\vspace{-1mm}
\begin{equation}
\label{eq:constrainedOpt}
\min_{\xv\in\setC} f(A\xv)
\end{equation}
for a constraint set $\setC$ can be directly written in the form \eqref{eq:primal} by using the indicator function of the constraint set as the penalization term~$g$. 
(The indicator function $\id_{\setC}$ of a set $\setC\subset \R^n$ is defined as $\id_{\setC}(\xv) := 0$ if $\xv\in\setC$ and $\id_{\setC}(\xv) := +\infty$ otherwise.)

\paragraph{The Partially Separable Case.}
A very important special case arises when one part of the objective becomes separable. 
Formally, this is expressed as $g(\xv) = \sum_{i=1}^{n} g_i(x_i)$ for univariate functions $g_i:\R\rightarrow\R$ for $i\in[n]$.
Nicely in this case, the conjugate of $g$ also separates as
$g^*(\yv) = \sum_i g_i^*(y_i)$. Therefore, the two optimization problems \eqref{eq:primal} and \eqref{eq:dual} write as\vspace{-2mm}
\begin{align}
    \calD(\xv) :=&\ f(A\xv )
    + \textstyle\sum_i g_i(x_i) \label{eq:primalS}\tag{SA}\\
    \calP(\wv) :=&\ f^*(\wv )
    + \textstyle\sum_i g_i^*(-\av_i^\top\wv) \ ,\label{eq:dualS}\tag{SB}\
\end{align}
where $\av_i \in \R^d$ denotes the $i$-th column of $A$.

Crucially in this case, the optimality conditions \eqref{eq:opt_g} and~\eqref{eq:opt_gstar} now become separable, that is\vspace{-2mm}
 \begin{subequations}
\begin{align}
-\av_i^\top \wv \in&\ \partial g_i(x_i) &&\forall i \label{eq:opt_gi} \ . 
\\
 x_i \in&\ \partial g_i^*(-\av_i^\top\wv) &&\forall i \label{eq:opt_gistar} \ .
\end{align}
  \end{subequations}
Note that the two other conditions \eqref{eq:opt_f} and \eqref{eq:opt_fstar} are unchanged in this case.
\section{Duality Gap and Certificates} \label{sec:duality_gap_certificates}
The duality gap for our problem structure provides an optimality certificate for our class of optimization problems. 
It will be the most important tool for us to provide guaranteed information about the optimal point (as in Section \ref{sec:info}), which will then be the foundation for the second step, to perform screening on the optimal point (as we will do in the later Sections \ref{sec:screening_constrained} and \ref{sec:screening_penalized}).

\subsection{Duality Gap Structure}

For the problem structure \eqref{eq:primal} and \eqref{eq:dual} as given by Fenchel-Rockafellar duality, the \emph{duality gap} for any pair of primal and dual variables $\xv\in\R^n$ and $\wv\in\R^d$ is defined as $\G(\wv,\xv) := \calD(\xv) + \calP(\wv)$. Non-negativity of the gap -- that is weak duality -- is satisfied by all pairs.

Most importantly, the duality gap acts as a certificate of approximation quality --- the true optimum values $\calD(\xv^*)$ and $-\calP(\wv^*)$ (which are both unknown) will always lie within the (known) duality gap.

\paragraph{The Gap Function.} 
For the special case of differentiable function $f$, we can study a simpler duality gap\vspace{-1mm}
\begin{equation}
\label{eq:gap}
\G(\xv) := \calD(\xv) + \calP(\wv(\xv))
\end{equation}
purely defined as a function of $\xv$, using the optimality relation \eqref{eq:opt_f}, i.e. $\wv(\xv) := \nabla f( A\xv )$.

\paragraph{The Wolfe-Gap Function.} 
For any constrained optimization problem \eqref{eq:constrainedOpt} defined over a bounded set $\domain$ and $\xv\in\domain$, the Wolfe gap function (also known as Hearn gap or Frank-Wolfe gap) is defined as the difference of $f$ to the minimum of its linearization over the same domain. 
Formally, 
\begin{equation}
\label{eq:wolfe_gap}
\GW(\xv) := \underset{\yv \in \domain}{\max}\; (A\xv-A\yv)^\top\nabla f(A\xv).
\end{equation}
It is not hard to see that the convenient Wolfe gap function is a special case of our above defined general duality gap $\G(\xv) := \calD(\xv) + \calP(\wv(\xv))$, for $g$ being the indicator function of the constraint set~$\domain$, and $\wv(\xv) := \nabla f(A\xv)$.
For more details, see Appendix \ref{app:eq_fw_gd}, or also \cite[Appendix~D]{LacosteJulien:2013ue}.

\subsection{Obtaining Information about the Optimal Points}
\label{sec:info}
As we have mentioned, any type of screening will crucially rely on first deriving safe knowledge about the unknown optimal points of our given optimization problem.
Here, we will use the duality gap to obtain such knowledge on the optimal points $\xv^\star\in\R^n$ and $\wv^\star\in\R^d$ of the respective optimization problems \eqref{eq:primal} and \eqref{eq:dual} respectively.
Proofs are provided in Appendix~\ref{app:optimal_point}.

Our first lemma shows how to bound the distance between any (feasible) current dual iterate and the solution $\wv^\star$ using standard assumptions on the objective functions. 

\begin{lemma}
\label{lem:stronglyconvex}
Consider the problem \eqref{eq:dual} with optimal solution $\wv^\star\in \R^d$. For $f$ being $\mu$-smooth%
, we have
\begin{equation}
\norm{\wv - \wv^\star}^2 \le \frac{2}{\mu}(f^*(\wv)-f^*(\wv^\star)) 
\end{equation}

\end{lemma}
The following corollary will be important to derive screening rules for penalized problems in Section~\ref{sec:screening_penalized}, as well as box-constrained problems (Section \ref{subsec:box_screening}).
\begin{corollary}
\label{cor:Dgaprestriction}
We consider the problem setup \eqref{eq:primal} and \eqref{eq:dual},
and assume $f$ is $\mu$-smooth.
Then
\begin{equation}
\norm{\wv - \wv^\star}^2 \le \frac{2}{\mu}\G(\xv).
\end{equation}
Here $\G(\xv)$ is the duality gap function as defined in equation \eqref{eq:gap}.
\end{corollary}

The following two results hold for general constrained optimization problems of the form \eqref{eq:constrainedOpt}, where~$g$ is the indicator function of a constraint set $\setC\subset\R^n$ and hence are useful for deriving screening rules for such problems.

\begin{lemma}
\label{lem:FWgaprestriction}
Consider problem \eqref{eq:primal} and assume that $f$ is $\mu$-strongly convex over a bounded set $\setC$. Then it holds that
\begin{equation}
\label{eq:FWgaprestriction}
\norm{A\xv-A\xv^\star}_2^2 \le \frac{1}{\mu}\GW(\xv),
\end{equation}
where $\xv^\star$ is an optimal solution and $\GW$ is the Wolfe-Gap function of $f$ over the bounded set $\setC$.
\end{lemma}

\begin{corollary}
\label{cor:restric:smooth}
Assuming $f$ is $L$-smooth as well as $\mu$-strongly convex over a bounded set $\setC$, we have
\begin{equation}
\norm{\nabla f(A\xv)-\nabla f(A\xv^\star)} \le \frac{L}{\sqrt{\mu}}\sqrt{\GW(\xv)}
\end{equation}
\end{corollary}

\section{Screening Rules for Constrained Problems}
\label{sec:screening_constrained}
In the following, we will develop screening rules for constrained optimization problems of the form~\eqref{eq:constrainedOpt}, by exploiting the structure of the constraint set for a variety of sparsity-inducing problems. First of all, we give a general lemma which we will be using in rest of the paper to derive screening rules when any of the function in \ref{eq:primal} and \ref{eq:dual} is indicator function.
\begin{lemma} \label{lem:general_constrained_lemma}
For general constrained optimization $\min_{\xv\in \setC} f(A\xv)$, the optimality condition~\eqref{eq:opt_g} gives rise to the following optimality rule at the optimal point:
\begin{equation}
\label{eq:general_constrained_rule}
(A\xv^\star)^\top \wv^\star = \min_{\zv \in \setC} (A\zv)^\top \wv^\star \vspace{-1mm}
\end{equation}
\end{lemma}
The above equation \eqref{eq:general_constrained_rule} also suggest that $\xv^\star = \argmin_{\zv \in \setC}(A\zv)^\top \wv^\star$. Lemma \ref{lem:general_constrained_lemma} is very crucial in further deriving screening rules for constrained optimization problem as well as norm penalized problems whose conjugate is indicator function of the dual norm. 
\subsection{Simplex Constrained Problems}
\label{subsec:simplex_screening}
Optimization over unit simplex $\Simplex := \{\xv\in\R^n \ | \ x_i \ge 0,\quad \sum_{i=1}^{n} x_i = 1 \}$ is a important class of constrained problems \eqref{eq:constrainedOpt}, as it includes optimization over any finite polytope. In this case, the columns of $A$ describe the vertices, and $\xv$ are barycentric coordinates representing the point $A\xv$. Formally, $g(\xv)$ is the indicator function of the unit simplex $\setC = \Simplex$ in this case.

The following two theorems provide screening rules for simplex constrained problems. 
We provide all proofs in Appendix \ref{app:simplex_screening}.

\begin{theorem}
\label{thm:simplexcondition}
For general simplex constrained optimization $\min_{\xv\in\Simplex} f(A\xv)$, the optimality condition~\eqref{eq:opt_g} gives rise to the following screening rule at the optimal point, for any $i \in [n]$
\begin{equation}
\label{eq:simplexcondition}
(\av_i - A\xv^\star)^\top\wv^\star > 0 
\ \Rightarrow \ x_i^\star = 0 \ .
\end{equation}
\end{theorem}

In the following Theorem \ref{thm:simplex_smooth} we now assume smoothness and strong convexity of function $f$ to provide screening rules for simplex problems, in terms of an arbitrary iterate $\xv$, without knowing $\xv^\star$.

\begin{theorem} \label{thm:simplex_smooth}
Let $f$ be $L$-smooth and $\mu$-strongly convex over the unit simplex $\setC = \Simplex$. Then for simplex constrained optimization $\min_{\xv\in\Simplex} f(A\xv)$ we have the following screening rule, for any $i \in [n]$
\begin{equation}
(\av_i -A\xv)^\top\nabla f(A\xv) > L\sqrt{\frac{\GW(\xv)}{\mu}}\norm{\av_i-A\xv} 
\ \Rightarrow \ x_i^\star=0 \ .
\end{equation}
\end{theorem}

Our general screening rules for simplex constrained problems as in Theorem \ref{thm:simplex_smooth} allows many practical implications.
For example, new screening rules for squared loss SVM and minimum enclosing ball problem come as a direct consequence.

\paragraph{Squared Hinge Loss SVM.}
The squared hinge-loss SVM problem in its dual form is formulated as
\begin{align} \label{eq:sq:loss:svm}
\min_{\xv \in \Simplex} \ \big[\ f(A\xv) := \tfrac12 \xv^\top A^\top A\xv \ \big] \vspace{-1mm}
\end{align}
over a unit simplex constraint $\xv\in\Simplex\subset\R^n$. Here for given data examples $\bar\av_1,\dots,\bar\av_n\in\R^d$ and corresponding labels $y_i\in\pm1$, the matrix $A$ collects the columns $\av_i=y_i\bar\av_i$, see e.g. \cite{Tsang:2005up}. %
We obtain the following novel screening rule for square loss SVM:

\begin{corollary} \label{cor:square_svm}
For the squared hinge loss SVM \eqref{eq:sq:loss:svm} we have the screening rule
\begin{align}
(\av_i - A\xv)^\top A\xv  &>\sqrt{\max_{i}\; (A\xv-\av_i)^\top A\xv} \,\|\av_i - A\xv\| \notag \\
& \Rightarrow \ x_i^\star = 0.
\end{align}
\end{corollary}

\paragraph{Minimum Enclosing Ball.}
The primal and dual for the minimum enclosing ball problem is given as the following pair of optimization formulations \eqref{eq:min_enclosing_ball} and \eqref{eq:dualball} respectively.
\begin{align} 
& \min_{\cv\in\R^d,r\in\R} \ \ \ r^2  \quad \text{s.t.} \ \ \norm{\cv-\av_i}^2_2\le r^2 \quad \forall i \in [n] \label{eq:min_enclosing_ball} \\
\qquad &~ \min_{\xv\in\Simplex\subset\R^n} \ \ \ \xv^\top A^\top A\xv + \cv^\top \xv \ , \label{eq:dualball}
\end{align}
where $\cv$ is a vector whose $i^{th}$ element $c_i$ is $-\av_i^\top \av_i$,
see for example \citep{Matousek:2007ub} or our Appendix~\ref{app:simplex_screening}.
\begin{corollary} \label{cor:min_enclosing_ball}
For the minimum enclosing ball problem \eqref{eq:min_enclosing_ball} we have the screening rule
\begin{align}\label{eq:ballcor}
(\bm{e}_i-\xv)^\top &(2A^\top A\xv+\cv) > \notag \\ &2\sqrt{\tfrac{1}{2}\max_i(\xv-\bm{e}_i)^\top (2A^\top A\xv+\cv)}\norm{\av_i-A\xv}  \notag
\\ \ \Rightarrow \ x_i^\star = 0 \ .
\end{align}
\end{corollary}
Our result improves upon the known rules by \cite{Kallberg:2014kd,Ahipasaoglu:2008il} by providing a broader selection criterion \eqref{eq:ballcor}.%

\subsection{$L_1$-Constrained Problems}\label{subsec:l1ball_screening}
$L_1$-constrained formulations are very widely used in order to induce sparsity in the variables. Here below we provide results for screening on general $L_1$-constrained problems, that is $\min_{\xv\in\setC} f(A\xv)$ for $\setC = L_1 \subset\R^n$ (or a scaled version of the $L_1$-ball).
Proofs are provided in Appendix~\ref{app:l1ball_screening}. 

\begin{theorem}\label{thm:l1_constrained_screening}
For general $L_1$-constrained optimization $\min_{\xv\in L_1} f(A\xv)$, the optimality condition~\eqref{eq:opt_g} gives rise to the following screening rule at the optimal point, for any $i \in [n]$
\begin{equation}
\label{eq:l1_constrained}
\abs{\av_i^\top \wv^\star}  + (A \xv^\star)^\top\wv^\star <0 
\ \Rightarrow \ x_i^\star  = 0\ .
\end{equation}
\end{theorem}

Using only a current iterate $\xv$ instead of an optimal point, we obtain screening for general smooth  and strongly convex function $f$:
\begin{theorem} \label{thm:l1_smooth_constrained}
Let $f$ be $L$-smooth and $\mu$-strongly convex over the $L_1$-ball.
Then for $L_1$-constrained optimization $\min_{\xv\in L_1} f(A\xv)$ we have the following screening rule, for any $i \in [n]$
\begin{align}
\abs{\av_i^\top \nabla f(A\xv)} + &(A\xv)^\top \nabla f(A\xv) \notag \\ +& L(\norm{\av_i}_2 + \norm{A\xv}_2)\sqrt{\frac{\GW(\xv)}{\mu}} < 0  \notag
\\ \Rightarrow \ \xv_i^\star = 0 \label{eq:l1_constrained_screening}
\end{align}
\end{theorem}

\subsection{Elastic Net Constrained Problems}
Elastic net regularization as an alternative to $L_1$ is often used in practice, and can outperform the Lasso, while still enjoying a similar sparsity of representation \cite{Zou:2005EN}. The elastic net is given by the expression
\[
\alpha \norm{\xv}_1 + \frac{(1 - \alpha) }{2} \norm{\xv}_2^2
\ .
\]
Here below we provide novel result for screening on general elastic net constrained problems, that is $\min_{\xv\in\setC} f(A\xv)$ for~$\setC$ being the elastic net constraint, or a scaled version of it.
Proofs are provided in Appendix \ref{app:elastic_net_screening}. 
\begin{theorem} \label{thm:elastic_net_constrained}
For general elastic net constrained optimization $\min_{\xv\in L_E} f(A\xv)$ where $L_E := \{\xv\in\R^n\ |\ \alpha \norm{\xv}_1 + \frac{(1 - \alpha) }{2} \norm{\xv}_2^2 \le 1\}$, the optimality condition~\eqref{eq:opt_g} gives rise to the following screening rule at the optimal point, for any $i \in [n]$
\begin{align*}
|\av_i^\top \wv^\star | + (A\xv^\star)^\top  \wv^\star \Big [ \frac{ \alpha  }{ 1 + \frac{(1 - \alpha) }{2} \|\xv^\star\|_2^2} \Big ] < 0  \Rightarrow x_i^\star = 0
\end{align*}
\end{theorem}

Using only a current iterate $\xv$ instead of an optimal point, we obtain screening for general smooth  and strongly convex function $f$:
\begin{theorem} \label{thm:elastic_smooth_constrained}
Let $f$ be $L$-smooth and $\mu$-strongly convex over the elastic net norm ball.
Then for elastic net constrained optimization $\min_{\xv\in L_E} f(A\xv)$ we have the following screening rule, for any $i \in [n]$
\begin{align}
\abs{\av_i^\top \nabla f(A\xv)} + &(A\xv)^\top \nabla f(A\xv) \big[ \frac{2\alpha}{3-\alpha} \big ] \notag \\ +& L(\norm{\av_i}_2 + \norm{A\xv}_2\big[ \frac{2\alpha}{3-\alpha} \big ])\sqrt{\frac{\GW(\xv)}{\mu}} < 0  \notag
\\ \Rightarrow \ \xv_i^\star = 0 \label{eq:l1_constrained_screening}
\end{align}
\end{theorem}
Note that both above results also recover the $L_1$ constrained case as a special case, when $\alpha \rightarrow 1$.

\subsection{Screening for Box Constrained Problems}\label{subsec:box_screening}
Box-constrained problems are important in several machine learning applications, including SVMs. After variable rescaling, w.l.o.g. we can assume the constraint set $\setC = \Box := \{\xv\in\R^n \ | \ 0 \le x_i \le 1 \}$.
We derive screening rules for predicting both if a variable will take the upper or lower constraint.
\begin{theorem} \label{thm:box_constrained}
Let $f$ be $L$-%
smooth.
Then for box-constrained optimization $\min_{\xv\in \Box} f(A\xv)$, we obtain the following screening rules, for any $i \in [n]$
\begin{align*} 
&\av_i^\top \nabla f(A\xv) -  \norm{\av_i}_2\sqrt{2L \G(\xv)}>0 
\ \Rightarrow \  x_i^\star =0 \ , \ \text{ and}\\
&\av_i^\top \nabla f(A\xv) +  \norm{\av_i}_2\sqrt{2L\,\G(\xv)} < 0 
\ \Rightarrow \  x_i^\star =1 \ . %
\end{align*}
\end{theorem}
Box constrained opptimization problems arise very often in machine learning probelm. Hinge loss SVM happens to one of many special cases of box-constrained optimizaion problem.
\paragraph{Hinge Loss SVM.}
The dual of the classical support vector machine with hinge loss, when not using a bias value, is a box-constrained problem. As a direct consequence of Theorem \ref{thm:box_constrained} we therefore obtain screening rules for SVM with hinge loss and no bias. 
The primal formulation of the SVM in this setting, for a regularization parameter $C>0$,  is%
\begin{equation}
\begin{aligned}
 &\underset{\wv\in\R^d,\epsilonv\in\R^n}{\min}  \tfrac{1}{2} \wv^\top\wv + C \one^\top\epsilonv \\
 &\text{s.t.}  \quad  \wv^\top\av_i \ge 1-\varepsilon_i  \ \ \forall i \in [n]  \\
\qquad  & \varepsilon_i \ge 0  \quad \forall i \in [n] \label{eq:primal:svm_hinge}
\end{aligned} 
\end{equation}
\begin{corollary}\label{cor:SVMhinge}
For SVM with hinge loss and no bias as given in \eqref{eq:primal:svm_hinge}, we have the screening rules
\begin{align*}
&\av_i^\top A\xv -\norm{\av_i}_2\sqrt{2\G(\xv)} >0 \ \Rightarrow \ x_i^\star = 0 \ , \ \text{ and}\\
&\av_i^\top A\xv +\norm{\av_i}_2\sqrt{2\G(\xv)} <0 \ \Rightarrow \ x_i^\star = C \ .
\end{align*}
where $\xv\in\R^n$ is any feasible dual point.
\end{corollary}
We get similar screening rules for hinge loss SVM as in \citep{Zimmert:2015ue} as well as in \citep{Shibagaki:2016ts}. The closest known result to our Corollary \ref{cor:SVMhinge} for screening in hinge loss SVM is given in \cite{Zimmert:2015ue} and \citep{Shibagaki:2016ts}.  The work of \cite{Zimmert:2015ue} also covers the kernelized SVM case, and improves the threshold given in our Corollary~\ref{cor:SVMhinge} by a constant of $\sqrt{2}$. In Appendix \ref{app:box_screening}, we show that our more general approach here can also be adjusted to gain this constant factor.
\section{Screening for Penalized Problems}
\label{sec:screening_penalized}
In this section we will develop screening methods for general penalized convex optimization problems of the form \eqref{eq:primal} and \eqref{eq:dual}. 
The cornerstone application are $L_1$ regularized problems, for which we now develop screening rules with general cost function $f$. 
We show in Appendix \ref{app:penalized_problem_l1} that our method can reproduce the screening rules of \cite{Ndiaye:2015wj} as special cases, whereas their method does not directly extend to general $f$.
Beyond $L_1$ problems, we also describe new screening rules for elastic net regularized problems, as well as the important case of structured norm regularized optimization. 

\subsection{$L_1$-Penalized Problems}

The next theorem describes a screening rule for general $L_1$-penalized problems, under a smoothness assumption on function $f$. 
Proofs for are given in Appendix \ref{app:penalized_problem_l1}. 
\begin{theorem} \label{thm:smooth_pen_lasso}
Consider an $L_1$-regularized optimization problem of the form\vspace{-1mm}
\begin{align}\label{eq:l1_regularized}
 \min_{\xv \in \R^n} \ \ f(A\xv) +  \lambda \norm{\xv}_1
\end{align}
\vspace{-1mm}
If $f$ is $L$-%
smooth, then the following screening rule holds for all $i\in[n]$:
\begin{align*}
\abs{\av_i^\top \nabla f(A\xv)} < \lambda - \norm{\av_i}_2\sqrt{2L\,\G(\xv)} 
\ \Rightarrow \ \xv_i^\star=0
\end{align*}
\end{theorem}

By careful observation of the expression in Theorem \ref{thm:smooth_pen_lasso}, it is easy to find a connection between our screening rule and the geometric sphere test method based screening \cite{Xiang:2014vi}. The general idea behind the sphere test is to consider the maximum value of the objective function in a spherical region which contains the optimal dual variable. We discuss this connection in more detail in Appendix~\ref{app:connection_sphere_test}.

Also, in Appendix \ref{app:penalized_problem_l1}, we discuss the special cases of squared loss regression and logistic loss regression with $L_1$ penalization. These results are presented in Corollaries~\ref{cor:pen_lasso} and \ref{cor:logistic_l1} as direct consequences of Theorem \ref{thm:smooth_pen_lasso}. 
Both of the corollaries can also be derived from the framework discussed in the paper \cite{Ndiaye:2015wj}.

\subsection{Elastic-Net Penalized Problems}
In the next corollary, we present a novel screening rule for the elastic net squared loss regression problem.
\begin{corollary} \label{cor:elastic_net_regr}
Consider the elastic net regression formulation
\begin{equation} \label{eq:elastic}
\min_{\xv\in\R^n} \ \tfrac{1}{2} \|A\xv - \bv\|_2^2 + \lambda_2  \|\xv\|_2^2 + \lambda_1 \|\xv\|_1
\end{equation}
The following screening rule holds for all $i\in[n]$:
\begin{align*}
&\abs{(\av_i^\top A + 2\lambda_2 \bm{e}_i^\top)\xv-  \av_i^\top \bv} < \\ &\quad  \qquad \lambda_1 - \sqrt{2(\av_i^\top \av_i + 2\lambda_2) \G(\xv)} \ \Rightarrow \ x_i^\star = 0.
\end{align*}
\end{corollary}

We also recover existing screening rules for elastic net regularized problem with more general objective $f$ using our frameworks, in Appendix \ref{app:penalized_problem_l1}, see Lemma \ref{lem:smooth_elastic} and Theorem \ref{thm:elasticnet_genral} which has been earlier derived in \citep{Shibagaki:2016ts}. In the proof, we derive screening rules from both the formulation \eqref{eq:primalS} and \eqref{eq:dualS} using optimality condition \eqref{eq:opt_gi} and \eqref{eq:opt_gistar} which is novel as well as help us to understand the property useful in deriving screening rules for elastic net penalized problems.

\subsection{Structured Norm Penalized Problems}
\label{subsec:l1_l2_genral}
Here in this section we present screening rules for non-overlapping group norm regularized problems. Group-norm regularization is widely used to induce sparsity in terms of groups of variables of the the solution of the optimization problem. 
The most prominent example is the group lasso ($\ell_2/\ell_1$-regularization). Here in this section we mostly discuss screening for general objectives with an $\ell_2/\ell_1$-regularization. Proofs are provided in Appendix \ref{app:structured_norm}.

\paragraph{Group Norm - $\ell_2/\ell_1$ Regularization.}  \label{par:l1_l2_group}
In the following, we use the notation $\{ \xv_1 \cdots \xv_G\}$ to express a vector $\xv$ as a partition of the groups of variables, such that $\xv^\top = \left[ \xv_1^\top, \xv_2^\top \cdots \xv_G^\top \right]$. Correspondingly, the matrix $A$ can be denoted as the concatenation of the respective columns $A = [A_1 \ A_2 \cdots A_G]$.

\begin{theorem} \label{thm:smooth_grp_lasso}
For $\ell_2/\ell_1$-regularized optimization problem of the form
\vspace{-3mm}
\begin{align*}
\ \quad  \min_{\xv} f(A\xv) + \sum_{g=1}^G \sqrt{\rho_g} \|\xv_g\|_2 \vspace{-2mm}
\end{align*} 
Assuming $f$ is $L$-smooth, then the following screening rule holds for all groups $g$:\vspace{-1mm}
\begin{align*}
\|A_g^\top \nabla f(A\xv) \|_2 + \sqrt{2L} \frobnorm{A_g} < \sqrt{\rho_g} \ \Rightarrow\  \xv_g^\star = 0 \ .
\end{align*}
\end{theorem}

\begin{corollary}\label{grplassbasic_spc}
\textbf{Group Lasso Regression with Squared Loss} - For the group lasso formulation
\vspace{-1mm} 
\begin{align*}
\min_{\xv} \ \tfrac{1}{2} \|  A \xv-\bv\|_2^2 + \lambda \sum_{g=1}^G \sqrt{\rho_g} \|\xv_g\|_2 
\end{align*} 
we have the following screening rule for all groups $g$:
\begin{align*}
\|A_g^\top (A\xv - \bv)\|_2 + \sqrt{2\G(\xv)}   \frobnorm{A_g} <\lambda \sqrt{\rho_g} \ \Rightarrow \ \xv_g^\star = 0 \ .
\end{align*}
\end{corollary}

Group lasso regression is widely used in applications as an working example case of structured norm penalization. 
The framework of \cite{Ndiaye:2015wj} does not directly provide screening rules for the group lasso, due to the fact that they require $f$ to be partially separable over the groups as well as special structural requirement of the formulation, 
in contrast to our more general Theorem~\ref{thm:smooth_grp_lasso}. Similarly, \cite{Lee:2014wb} is also restricted to least-squares $f$ objective.

\vspace{-0.2cm}
\section{Illustrative Experiments} \label{sec:exp}
\vspace{-0.2cm}
While the contribution of our paper is on the theoretical generality and the collection of new screening applications, 
we will still briefly illustrate the performance of some of the proposed screening algorithms, for the classical examples of simplex constrained and $L_1$-constrained problems.
We compare the fraction of active variables and the Wolfe-Gap function as optimization algorithm progress. %

We consider the optimization problem of the form $\underset{\xv \in \ball_{L_1}}\min \|A\xv - \bv\|_2^2$. $\ball_{L_1}$ is a scaled $L_1$-ball with radius $35$. $A \in \mathbb{R}^{3000\times600}$  is a random Gaussian matrix and a noisy measurement
$\bv = A\xv^\star$ where $\xv^\star$ is a sparse vector of $+1$ and $-1$ with only $70$ non zeros entries. We solve the above optimization problem using the Frank-Wolfe algorithm (pair-wise variant, see \cite{LacosteJulien:2015wj}).
Before putting this optimization problem into the solver %
we convert this problem into the  barycentric representation which is $\underset{\xv_{\Simplex} \in \Simplex}\min \|A_{\Simplex}\xv_{\Simplex} - \bv\|_2^2$. The relation between the transformed variable and original variable can be given by $A_{\Simplex} = [A\,|-A]$ and $\xv = [I_n\,| - I_n]\xv_{\Simplex}$. For more details see \citep{Jaggi:2014co}. 
\begin{table}[H]
\centering
\begin{tabular}{ | c | c | c |}
\hline
\textbf{Dataset/ } & \textbf{No Screening}&\textbf{Screening} \\ 
\textbf{No. of Samples} & \textbf{(Simplex)}&\textbf{(Simplex)}\\ \hline
 \textit{Synth1} 5000& 13.1 sec &11.7sec\\ \hline
\textit{Synth2} 10000 & 28.3 sec & 23.1 sec\\ \hline
\textit{RCV1} 20242 & 18.6 min  &13.5 min \\ \hline
\textit{news20B} 19996 & 33.4 min  &25.2 min \\ \hline
\end{tabular} 
\caption{Simplex-constrained screening, clock time}
\label{tab:tab1}
\end{table}

\begin{figure}
\includegraphics[width=8.5cm, height = 5cm]{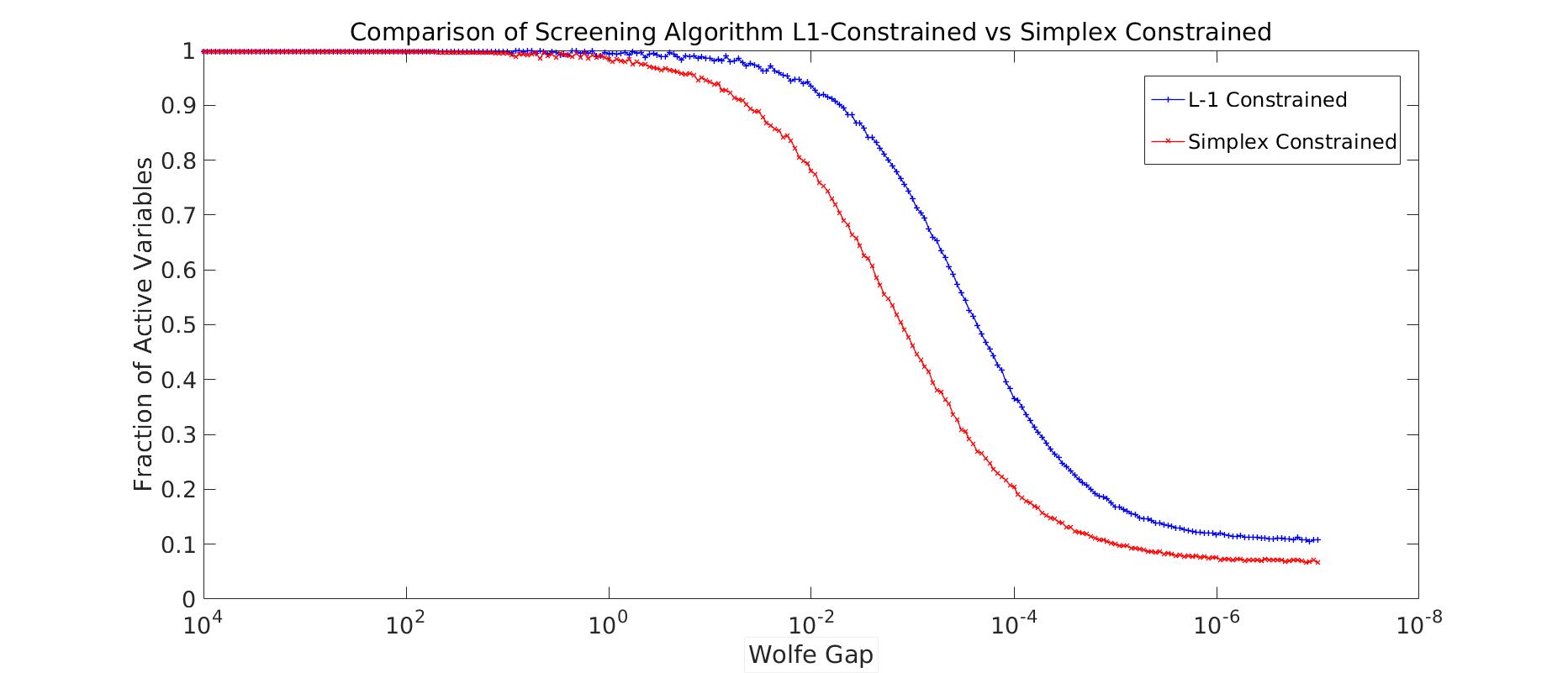} 
  \vspace{-25pt}
  \caption{Simplex- vs $L_1$-constrained Screening}
  \label{fig:fig1}
\end{figure}
\vspace{-0.7cm}
\begin{table}[H]
\centering
\begin{tabular}{ | c | c | c |}
\hline
\textbf{Dataset/ } & \textbf{No Screening}&\textbf{Screening} \\ 
\textbf{No. of Samples} & \textbf{($\ell_1$-constr.)}&\textbf{($\ell_1$-constr.)}\\ \hline
\textit{Synth1} 5000&13.1 &12.2 sec\\ \hline
\textit{Synth2} 10000 &  28.3 sec &24.7  sec\\ \hline
\textit{RCV1} 20242 & 18.6 min &14.9  min\\ \hline
\textit{news20B} 19996 & 33.4 min  &27.1 min \\ \hline
\end{tabular}
\caption{$L_1$-constrained screening, clock time}
\label{tab:tab2}
\end{table}
\vspace{-0.5cm}
Now we apply our Theorems \ref{thm:l1_smooth_constrained} and \ref{thm:simplex_smooth} on variable of $\xv$ and~$\xv_{\Simplex}$ respectively to screen, in order to compare the two alternative screening approaches on the same problem. Note that the Wolfe gap is identical in both parameterizations, for any~$\xv$.
One important point to note here is that dimension of $\xv_{\Simplex}$ is the double of the dimension of $\xv$, and any $L_1$-coordinate value $x_i$ is zero if and only if both ``duplicate'' variables $x_{\Simplex,i}$ and $x_{\Simplex,n+i}$ are zero, where $n$ is the dimensionality of~$\xv$. 

Therefore, the simplex variant (with more variables) performs a more fine-grained variant of screening, where we can screen each of the sign patterns separately for each variable. In Fig \ref{fig:fig1}, the blue curve illustrates the screening efficiency for the $L_1$-constrained screening case, while the red curve illustrate simplex constrained screening. 
 Our theorems \ref{thm:l1_smooth_constrained} and  \ref{thm:simplex_smooth} are well in line with the phenomena in Fig \ref{fig:fig1}. 
For the $L_1$-constrained case, the screening starts relatively at later stage than simplex case due to the fact that in Equation \eqref{eq:l1_constrained_screening}, two out of three terms are absolute values of some quantity and hence it is very tough to compensate both of them by the third quantity, in order for the entire sum to become negative. %
Hence in the beginning this rule can often be ineffective. 
As algorithm progresses, the duality gap becomes smaller and screening starts but at the same time the gradient (and therefore gap) also starts to decay which brings the trade-off shown in the plot. For both variants, screening becomes slow towards the end. 

We also report the time taken to reach a duality gap of $10^{-7}$ with both the approaches mentioned above (simplex constrained and $L_1$-constrained) on for different datasets. The first two datasets (\textit{Synth1} and \textit{Synth2}) are generated under the same setting described earlier but \textit{Synth1} with 5000 samples and \textit{Synth2} with 10000 samples. \textit{RCV1} is a real world dataset having $20,242$ samples and $47,236$ data dimensions. \textit{news20Binary} is also a real world dataset having $19,996$ entries and $ 1,355,191$ dimensions. Below in Tables \ref{tab:tab1} and \ref{tab:tab2}, we describe the running time of the optimization methods to reach a duality gap threshold of $10^{-7}$  with or without screening. On \textit{RCV1} dataset we try the feature learning with $L_1$-norm ball constraint of $200$ and on \textit{news20Binary} we use $L_1$-norm ball constraint of $35$. In the case of \textit{RCV1} and \textit{news20Binary}, $A$ is the data matrix and $\bv$ is the label of each instance in the dataset. From Tables~\ref{tab:tab1} and~\ref{tab:tab2} it is also evident that simplex screening rule is more tighter than the $L_1$-constrained screening rule.\\

\vspace{-0.8cm}
\section{Discussion}
\vspace{-0.3cm}
We have presented a unified way to derive screening rules for general constrained and penalized optimization problems. 
For both cases, our framework crucially utilizes the structure of piece-wise linearity of the problem at hand.
For the constrained case, we showed that screening rules follow from the piece-wise linearity of the boundary of the constraint set.
\\
The crucial property is that at non-differentiable boundary points, the normal cone -- i.e. the sub-differential of the indicator function of the constraint set -- becomes a relatively large set. Under moderate assumptions on the objective function, we are able to guarantee that also the gradient of an optimal point must lie in this same cone region, leading to screening.
\\
On the other hand for penalized optimization problems, we are able to derive screening rules from either piece-wise linearity of the penalty function, or as well from exploiting piece-wise linearity of the constraint set arising from the dual (conjugate) of the penalty function.

\paragraph{Acknowledgements.}
The authors would like to thank Julia Wysling, Rohit Babbar and Peter B\"uhlmann for valuable discussions.

\clearpage
{\small
\bibliographystyle{abbrvnat}
\bibliography{bibliography-jakob,bibliography-martin}

\begin{thebibliography}{25}
\providecommand{\natexlab}[1]{#1}
\providecommand{\url}[1]{\texttt{#1}}
\expandafter\ifx\csname urlstyle\endcsname\relax
  \providecommand{\doi}[1]{doi: #1}\else
  \providecommand{\doi}{doi: \begingroup \urlstyle{rm}\Url}\fi

\bibitem[Ahipasao{\u{g}}lu et~al.(2008)Ahipasao{\u{g}}lu, Sun, and
  Todd]{Ahipasaoglu:2008il}
S.~D. Ahipasao{\u{g}}lu, P.~Sun, and M.~Todd.
\newblock {Linear Convergence of a Modified Frank{\textendash}Wolfe Algorithm
  for Computing Minimum-Volume Enclosing Ellipsoids}.
\newblock \emph{Optimization Methods and Software}, 23\penalty0 (1):\penalty0
  5--19, Feb. 2008.

\bibitem[Bauschke and Combettes(2011)]{Bauschke:2011ik}
H.~H. Bauschke and P.~L. Combettes.
\newblock \emph{{Convex Analysis and Monotone Operator Theory in Hilbert
  Spaces}}.
\newblock CMS Books in Mathematics. Springer New York, New York, NY, 2011.

\bibitem[Borwein and Zhu(2005)]{Borwein:2005ge}
J.~M. Borwein and Q.~Zhu.
\newblock \emph{{Techniques of Variational Analysis and Nonlinear
  Optimization}}.
\newblock Canadian Mathematical Society Books in Math, Springer New York, 2005.

\bibitem[Boyd and Vandenberghe(2004)]{Boyd:2004uz}
S.~P. Boyd and L.~Vandenberghe.
\newblock \emph{{Convex optimization}}.
\newblock Cambridge University Press, 2004.

\bibitem[Ghaoui et~al.(2010)Ghaoui, Viallon, and Rabbani]{Ghaoui:2010tu}
L.~E. Ghaoui, V.~Viallon, and T.~Rabbani.
\newblock {Safe Feature Elimination for the LASSO and Sparse Supervised
  Learning Problems}.
\newblock \emph{arXiv.org}, Sept. 2010.

\bibitem[Jaggi(2014)]{Jaggi:2014co}
M.~Jaggi.
\newblock {An Equivalence between the Lasso and Support Vector Machines}.
\newblock In \emph{Regularization, Optimization, Kernels, and Support Vector
  Machines}, pages 1--26. Chapman and Hall/CRC, Oct. 2014.

\bibitem[Kakade et~al.(2009)Kakade, Shalev-Shwartz, and Tewari]{Kakade:2009wh}
S.~M. Kakade, S.~Shalev-Shwartz, and A.~Tewari.
\newblock {On the duality of strong convexity and strong smoothness: Learning
  applications and matrix regularization}.
\newblock Technical report, Toyota Technological Institute - Chicago, USA,
  2009.

\bibitem[K{\"a}llberg and Larsson(2014)]{Kallberg:2014kd}
L.~K{\"a}llberg and T.~Larsson.
\newblock {Improved Pruning of Large Data Sets for the Minimum Enclosing Ball
  Problem}.
\newblock \emph{Graphical Models}, July 2014.

\bibitem[Lacoste-Julien and Jaggi(2015)]{LacosteJulien:2015wj}
S.~Lacoste-Julien and M.~Jaggi.
\newblock {On the Global Linear Convergence of Frank-Wolfe Optimization
  Variants}.
\newblock In \emph{NIPS 2015 - Advances in Neural Information Processing
  Systems 28}, 2015.

\bibitem[Lacoste-Julien et~al.(2013)Lacoste-Julien, Jaggi, Schmidt, and
  Pletscher]{LacosteJulien:2013ue}
S.~Lacoste-Julien, M.~Jaggi, M.~Schmidt, and P.~Pletscher.
\newblock {Block-Coordinate Frank-Wolfe Optimization for Structural SVMs}.
\newblock In \emph{ICML 2013 - Proceedings of the 30th International Conference
  on Machine Learning}, 2013.

\bibitem[Lee and Xing(2014)]{Lee:2014wb}
S.~Lee and E.~P. Xing.
\newblock {Screening Rules for Overlapping Group Lasso}.
\newblock \emph{arXiv}, Oct. 2014.

\bibitem[Liu et~al.(2014)Liu, Zhao, Wang, and Ye]{Liu:2014uz}
J.~Liu, Z.~Zhao, J.~Wang, and J.~Ye.
\newblock {Safe Screening with Variational Inequalities and Its Application to
  Lasso}.
\newblock In \emph{ICML 2014 - Proceedings of the 31st International Conference
  on Machine Learning}, pages 289--297, 2014.

\bibitem[Matou{\v s}ek and G{\"a}rtner(2007)]{Matousek:2007ub}
J.~Matou{\v s}ek and B.~G{\"a}rtner.
\newblock \emph{{Understanding and using linear programming}}.
\newblock Springer, 2007.

\bibitem[Ndiaye et~al.(2015)Ndiaye, Fercoq, Gramfort, and
  Salmon]{Ndiaye:2015wj}
E.~Ndiaye, O.~Fercoq, A.~Gramfort, and J.~Salmon.
\newblock {GAP Safe screening rules for sparse multi-task and multi-class
  models}.
\newblock In \emph{NIPS 2015 - Advances in Neural Information Processing
  Systems 28}, pages 811--819, 2015.

\bibitem[Ogawa et~al.(2013)Ogawa, Suzuki, and Takeuchi]{Ogawa:2013ul}
K.~Ogawa, Y.~Suzuki, and I.~Takeuchi.
\newblock {Safe Screening of Non-Support Vectors in Pathwise SVM Computation}.
\newblock In \emph{ICML}, pages 1382--1390, 2013.

\bibitem[Ogawa et~al.(2014)Ogawa, Suzuki, Suzumura, and Takeuchi]{Ogawa:2014vw}
K.~Ogawa, Y.~Suzuki, S.~Suzumura, and I.~Takeuchi.
\newblock {Safe Sample Screening for Support Vector Machines}.
\newblock \emph{arXiv.org}, Jan. 2014.

\bibitem[Olbrich(2015)]{Olbrich:2015kq}
J.~Olbrich.
\newblock {Screening Rules for Convex Problems}.
\newblock Master's thesis, ETH Z{\"u}rich, 2015.

\bibitem[Shibagaki et~al.(2016)Shibagaki, Karasuyama, Hatano, and
  Takeuchi]{Shibagaki:2016ts}
A.~Shibagaki, M.~Karasuyama, K.~Hatano, and I.~Takeuchi.
\newblock {Simultaneous Safe Screening of Features and Samples in Doubly Sparse
  Modeling}.
\newblock In \emph{ICML 2016 - Proceedings of the 33th International Conference
  on Machine Learning}, pages 1577--1586, 2016.

\bibitem[Smith et~al.(2015)Smith, Forte, Jaggi, and Jordan]{Smith:2015ua}
V.~Smith, S.~Forte, M.~Jaggi, and M.~I. Jordan.
\newblock {L$_{1}$-Regularized Distributed Optimization: A
  Communication-Efficient Primal-Dual Framework}.
\newblock \emph{arXiv cs.LG}, 2015.

\bibitem[Tsang et~al.(2005)Tsang, Kwok, and Cheung]{Tsang:2005up}
I.~W. Tsang, J.~T. Kwok, and P.-M. Cheung.
\newblock {Core Vector Machines: Fast SVM Training on Very Large Data Sets}.
\newblock \emph{Journal of Machine Learning Research}, 6:\penalty0 363--392,
  2005.

\bibitem[Wang et~al.(2013)Wang, Lin, Gong, Wonka, and Ye]{Wang:2013uq}
J.~Wang, B.~Lin, P.~Gong, P.~Wonka, and J.~Ye.
\newblock {Lasso Screening Rules via Dual Polytope Projection}.
\newblock In \emph{NIPS 2014 - Advances in Neural Information Processing
  Systems 27}, 2013.

\bibitem[Wang et~al.(2014)Wang, Zhou, Liu, Wonka, and Ye]{Wang:2014un}
J.~Wang, J.~Zhou, J.~Liu, P.~Wonka, and J.~Ye.
\newblock {A Safe Screening Rule for Sparse Logistic Regression}.
\newblock In \emph{NIPS 2014 - Advances in Neural Information Processing
  Systems 27}, pages 1053--1061, 2014.

\bibitem[Xiang et~al.(2014)Xiang, Wang, and Ramadge]{Xiang:2014vi}
Z.~J. Xiang, Y.~Wang, and P.~J. Ramadge.
\newblock {Screening Tests for Lasso Problems}.
\newblock \emph{arXiv.org}, May 2014.

\bibitem[Zimmert et~al.(2015)Zimmert, de~Witt, Kerg, and Kloft]{Zimmert:2015ue}
J.~Zimmert, C.~S. de~Witt, G.~Kerg, and M.~Kloft.
\newblock {Safe screening for support vector machines}.
\newblock In \emph{NIPS Workshop on Optimization for Machine Learning}, pages
  1--5, Dec. 2015.

\bibitem[Zou and Hastie(2005)]{Zou:2005EN}
H.~Zou and T.~Hastie.
\newblock Regularization and variable selection via the elastic net.
\newblock \emph{Journal of the Royal Statistical Society: Series B (Statistical
  Methodology)}, 67\penalty0 (2):\penalty0 301--320, 2005.

\end{thebibliography}
}

\clearpage
\onecolumn
\appendix

\section{Primal Dual Structure (Section \ref{sec:primaldual})}
\label{app:primaldual}

The relation of our primal and dual problems \eqref{eq:primal} and \eqref{eq:dual} is standard in convex analysis, and is a special case of the concept of Fenchel Duality. 
Using the combination with the linear map $A$ as in our case, the relationship is called \emph{Fenchel-Rockafellar Duality}, see e.g. \cite[Theorem 4.4.2]{Borwein:2005ge} or \cite[Proposition 15.18]{Bauschke:2011ik}.
For completeness, we here illustrate this correspondence with a self-contained derivation of the duality.
\begin{proof}

Starting with the original formulation \eqref{eq:primal}, we introduce a helper variable vector $\vv\in \R^d$ representing $\vv =A\alphav$. Then optimization problem~\eqref{eq:primal} becomes:
\begin{equation}
\label{eq:constrainedprimal}
\min_{\alphav \in\R^n} \quad  f(\vv) + g( \alphav) \quad \text{such that} \ \vv =A\alphav \, .
\end{equation}
Introducing Lagrange multipliers $\wv \in \R^d$,  the Lagrangian is given by:
$$ L(\alphav, \vv; \wv) := f(\vv) +   g(\alphav) + \wv^\top\left(A\alphav-\vv\right) \, .$$
The dual problem of \eqref{eq:primal} follows by taking the infimum with respect to both $\alphav$ and $\vv$:
\begin{align}
\inf_{\alphav, \vv} L(\wv, \alphav, \vv) & =   \inf_{\vv} \left\{ f(\vv) - \wv^\top \vv \right\} + \inf_{\alphav} \left\{ g(\alphav) +  \wv^\top A\alphav\right\} \notag \\
& =  - \sup_{\vv} \left\{  \wv^\top \vv - f(\vv) \right\}- \sup_{\alphav} \left\{(-\wv^\top A)\alphav -  g(\alphav) \right\} \label{eq:optimlity_proof}\\
& = - f^*(\wv) - g^*(-A^\top \wv) \label{eq:Lagrangian}\, .
\end{align}

We change signs and turn the  maximization of the dual problem \eqref{eq:Lagrangian} into a minimization and thus we arrive at the dual formulation $\eqref{eq:dual}$ as claimed:
\[
    \min_{\wv \in \R^d} \quad \Big[ \ 
    \calP(\wv) := f^*(\wv) + g^*(-A^\top \wv) \ \Big] \, .
\]
 
\paragraph{The Partially Separable Case.}
For $g(\xv)$ is separable, i.e. $g(\xv) = \sum_{i=1}^{n} g_i(x_i)$ for univariate functions $g_i:\R\rightarrow\R$ for $i\in[n]$, the primal-dual structure remains the separable.  
In this case, the conjugate of $g$ also separates as
$g^*(\yv) = \sum_i g_i^*(y_i)$. Therefore, in terms of the the primal-dual structure \eqref{eq:primal} and \eqref{eq:dual} we obtain the separable special case \eqref{eq:primalS} and \eqref{eq:dualS}. 
\end{proof}

\paragraph{Optimality Conditions.}
The first-order optimality conditions follow from the standard definition of the conjugate functions in the Fenchel dual problem, see also e.g. \cite{Borwein:2005ge,Bauschke:2011ik}.
\begin{proof}
The first-order optimality conditions for our pair of vectors $\wv\in \R^d, \xv \in \R^n$ in problems~\eqref{eq:primal} and \eqref{eq:dual} are given by equations \eqref{eq:opt_f}, \eqref{eq:opt_g}, \eqref{eq:opt_fstar} and \eqref{eq:opt_gstar}.
The proof directly comes from equation \eqref{eq:optimlity_proof} by separately writing optimizing conditions for two expressions $\wv^\top \vv - f(\vv)$ and $(-\wv^\top A)\alphav -  g(\alphav)$ in equation \eqref{eq:optimlity_proof}.

Crucially in the partially separable case, the optimality conditions \eqref{eq:opt_g} and~\eqref{eq:opt_gstar} become separable.
Comparing the expressions \eqref{eq:primalS} and \eqref{eq:primal}, we see that 
$g(\xv) = \sum_i g_i(x_i)$ and hence 
\[
g^*(\xv) = \sum_i g_i^*(x_i)
\]
Hence by applying \eqref{eq:opt_g} and \eqref{eq:opt_gstar} we obtain the separable optimality conditions \eqref{eq:opt_gi} and \eqref{eq:opt_gistar}.
\end{proof}

\section{Duality Gap and Objective Function Properties} \label{app:duality_gap_certificates}

\subsection{Wolfe Gap as a Special Case of Duality Gap} \label{app:eq_fw_gd}
\begin{proof}
To see this as a special case of general duality gap of the problem formulation, we consider the constraint as indicator function of set $\domain$ such that $g(\xv) = \id_{\domain}(\xv)$. Now from the definition of the Wolfe gap function 
$$\GW(\xv) := \underset{\yv \in \domain}{\max}\; (A\xv-A\yv)^\top\partial f(A\xv)$$
Here $\partial f(A\xv)$ is an arbitrary subgradient of $f$ at the candidate position $\xv$, and $\id_{\domain}^*(\yv) := \sup_{\sv \in \domain}\langle \sv,\yv \rangle$ is the support function of $\domain$.
Now writing the general duality gap $\G(\xv)$ as 
\begin{align*}
\G(\xv) &:= \calD(\xv) + \calP(\wv(\xv))  \\
&:=  f(A\xv) + \id_{\domain}(\xv) + f^*(\wv(\xv) ) + \id_{\domain}^*(-(A^\top\wv(\xv)))
\end{align*}
the last term disappears since we assumed $\xv\in\domain$. Using the definition of the Fenchel conjugate, one has the Fenchel-Young inequality, i.e. $$f^*(\wv) := \max_{\uv\in\R^d} \wv^\top \uv - f(\uv) \ \Rightarrow \ f^*(\wv) +  f(\uv) \ge \wv^\top \uv$$
The above holds with equality if $\wv$ is chosen as a subgradient of $f$ at $\uv=A\xv$.
Therefore, using our first-order optimality mapping $\wv(\xv) := \partial f(A\xv)$, 
we have $$\G(\xv) = (A\xv)^\top \partial f(A\xv) + \id_{\domain}^*(-(A^\top\wv(\xv))) = \GW(\xv)$$
This derivation is adapted from \cite[Appendix D]{LacosteJulien:2013ue}.
\end{proof}

\subsection{Obtaining Information about the Optimal Points} \label{app:optimal_point}

\begin{lemma}[Conjugates of Indicator Functions and Norms]~\vspace{-1mm}
\label{lem:conjugates}
\begin{enumerate}
\item[i)] The conjugate of the indicator function $\id_{\setC}$ of a set $\setC\subset \R^n$ (not necessarily convex) is the support function of the set $\setC$, that is $\quad \id_{\setC}^*(\xv) = \sup_{\sv\in\setC} \langle \sv,\xv\rangle$
\vspace{-1mm}
\item[ii)] The conjugate of a norm is the indicator function of the unit ball of the dual norm.
\end{enumerate}
\end{lemma}
\begin{proof}
\cite[Example 3.24 and 3.26]{Boyd:2004uz}
\end{proof}

\begin{lemma} \label{lem:smooth_convex_condition}
Assume  that $f$ is  a closed  and  convex  function then  $f^*$ is $\mu$-strongly convex with respect to a norm $\norm{\cdot}$ if and only if $f$ is $1/\mu$-Lipschitz gradient with respect to dual norm $\norm{\cdot}_*$.
\end{lemma}
\begin{proof}
\cite[Theorem 3]{Kakade:2009wh}
\end{proof}

\begin{proof}[Proof of Lemma \ref{lem:stronglyconvex}]
From the definition of $\mu$-strongly convex function, we know that
\begin{align*}
 f^*(\wv) &\ge f^*(\wv^\star) + (\wv-\wv^\star)^\top\nabla f^*(\wv^\star) + \frac{\mu}{2} \|\wv - \wv^\star\|_2^2\\
 &\ge f^*(\wv^\star) + \frac{\mu}{2} \|\wv - \wv^\star\|_2^2 
\end{align*}
The first inequality follows directly by using the first order optimality condition for $\wv^\star$ being optimal. For any optimal point $\wv^\star$ and another feasible point $\wv$, $$ (\wv-\wv^\star)^\top\nabla f^*(\wv^\star) \ge 0.$$ \\
Hence,$ \qquad\norm{\wv^\star - \wv}_2^2 \le \frac{2}{\mu}(f^*(\wv)-f^*(\wv^\star))$
\end{proof}

\begin{proof}[Proof of Corollary \ref{cor:Dgaprestriction}]
This statement directly comes from \eqref{lem:stronglyconvex} and the definition of the duality gap. By definition we know that the true optimum values $\calD(\xv^*)$ and $-\calP(\wv^*)$ respectively for primal \eqref{eq:primal} and dual formulation \eqref{eq:dual}  will always lie within the duality gap which implies $$ \G(\xv)  \ge \calP(\wv) - \calP(\wv^\star) $$ 
By equation \eqref{eq:dual}, we know that $ \ \qquad \calP(\wv) = f^*(\wv) + g^*(-A^\top\wv^\star)$ \\
Now since $f^*$ is $\mu$-strongly convex function and $g^*$ is convex hence,
\begin{align}
f^*(\wv) \ge f^*(\wv^\star) + \nabla f^*(\wv^\star)^\top(\wv - \wv^\star) + \frac{\mu}{2}\norm{\wv - \wv^\star}_2^2 \label{eq:f_star_strng_convex} \\
g^*(-A^\top \wv)\ge g^*(-A^\top\wv^\star) + \nabla g^*(-A^\top \wv^\star)^\top(-A^\top\wv + A^\top\wv^\star)  \label{eq:g_star_strng_convex}  
\end{align}

Hence by adding equation \eqref{eq:f_star_strng_convex} and \eqref{eq:g_star_strng_convex}, we get
\[
\calP(\wv) \ge \calP(\wv^\star) + (\nabla f^*(\wv^\star) -A \nabla g^*(-A^\top\wv^\star)) ^\top(\wv - \wv^\star)+ \frac{\mu}{2}\norm{\wv - \wv^\star}_2^2
\]
\[
\Rightarrow \calP(\wv) \ge \calP(\wv^\star) + \nabla \calP(\wv^\star)^\top(\wv - \wv^\star)+ \frac{\mu}{2}\norm{\wv - \wv^\star}_2^2
\]
At optimal point $\wv^\star$, $\nabla  \calP(\wv^\star)^\top(\wv - \wv^\star) \ge 0$. \\
Hence,
\[
\G(\xv)  \ge \calP(\wv) - \calP(\wv^\star) \ge \frac{\mu}{2}\norm{\wv - \wv^\star}_2^2
\]
\end{proof}

\begin{proof}[Proof of Lemma \ref{lem:FWgaprestriction}]
From the definition of $\mu$-strong convexity of $f$ and using optimality condition,
\begin{eqnarray}
\mu\norm{A\xv-A\xv^\star}^2 & \le & (A\xv-A\xv^\star)^\top(\nabla f(A\xv)-\nabla f(A\xv^\star)) \label{eq:definition_strng_convex}\\
& \le & (A\xv-A\xv^\star)^\top\nabla f(A\xv) \label{eq:optimality_cond_f}\\
& \le & \GW(\xv)\label{eq:definition_wolfe_gap}
\end{eqnarray}
Equation \eqref{eq:definition_strng_convex} comes from the definition of $\mu$-strong convexity.\\
Equation \eqref{eq:optimality_cond_f} is first order optimality condition for $\xv^\star$ being optimal which implies $$(A\xv-A\xv^\star)^\top\nabla f(A\xv^\star) \ge 0$$
The inequality \eqref{eq:definition_wolfe_gap} follows by the definition of the gap function given in \eqref{eq:wolfe_gap}.
\end{proof}

\begin{proof}[Proof of Corollary \ref{cor:restric:smooth}]
This comes by definition of $L$-smooth functions and Lemma \ref{lem:FWgaprestriction}.
From the definition, 
\begin{align*}
\norm{\nabla f(A\xv)-\nabla f(A\xv^\star)} &\le L \norm{A\xv-A\xv^\star} \\
& \le  \frac{L}{\sqrt{\mu}}\sqrt{\GW(\xv)}
\end{align*}
Second inequality directly comes from Lemma \ref{lem:FWgaprestriction}.
\end{proof}

\section{Screening on Constrained Problems}
\label{app:constrained}

\begin{lemma} \label{lem:subgrad_ind}
Let $\setC$ be a convex set, and $\id_\setC$ be its indicator function, then
\begin{enumerate}
\item For $\xv \notin \setC$, $\partial \id_\setC(\xv) = \emptyset$
\item For $\xv \in \setC$, we have that $\wv \in  \partial \id_\setC(\xv)$ if $ \wv^\top(\zv-\xv) \le 0 \quad \forall \zv \in \setC$
\end{enumerate}
\end{lemma}

\begin{proof}
Let $\domain \subseteq \R^n$ be a closed convex set. Then subgradient of indicator function $\id_{\domain}(\xv)$ at $\xv$ will be vectors $\uv$ which satisfy
\begin{align}
&\id_{\domain}(\zv) \ge \id_{\domain}(\xv) + \uv^\top(\zv - \xv) \quad \forall \zv \in \dom(\id_{\domain}) \notag \\
&\Rightarrow \id_{\domain}(\zv) \ge \id_{\domain}(\xv) + \uv^\top(\zv - \xv) \quad \forall \zv \in \R^n \label{app:eq:sub_diff}
\end{align}
If $\intr(\domain)$ represents the interior of the set $\domain$ such that it contains $n$-dimensional ball of radius $r >0$, and $\mathop{Bd}(\domain)$ represents boundary of the set $\domain$. Now we have to assume various cases for proving Lemma~\ref{lem:subgrad_ind}.
\begin{enumerate}
\item[Case 1] We evaluate Equation \eqref{app:eq:sub_diff} when $\xv \in \intr(\domain)$. Equation \eqref{app:eq:sub_diff} becomes $$\id_{\domain}(\zv)  \ge \uv^\top(\zv - \xv) \quad \forall \zv \in \R^n  $$
Now since the above equation is satisfied for all $\zv \in \R^n$ , we assume  $\zv \in \intr(\domain)$ such that $(\zv - \xv)$ can be anywhere in the ball. Hence $\uv$ needs to be $0$ in this case.
\item[Case 2] In this case we assume $\xv \in \mathop{Bd}(\domain)$. That gives $$\id_{\domain}(\zv)  \ge \uv^\top(\zv - \xv) \quad \forall \zv \in \R^n  $$ 
If we take $\zv \in \domain$ then $\uv$ satisfies $\uv^\top(\zv-\xv) \le 0 \quad \forall \zv \in \setC$ \\
If $\zv \not\in  \domain$ then $\uv$ can take all the value. Hence taking intersection, $\uv$ satisfies $$\uv^\top(\zv-\xv) \le 0 \quad \forall \zv \in \setC$$ 
\item[Case 3] When we assume $\xv \not \in \domain$, we get $$\id_{\domain}(\zv) \ge +\infty + \uv^\top(\zv - \xv) \quad \forall \zv \in \R^n $$ 
 If we again take $\zv \in \domain$ then no finite $\uv$ can satisfy the equation $\id_{\domain}(\zv) \ge +\infty + \uv^\top(\zv - \xv) \quad \forall \zv \in \domain $  because $\id_{\domain}(\zv) = 0 \ \textrm{ if} \ \ \zv \in \domain$. \\
 And if $ \zv \not \in \domain \ \Rightarrow\  \id_{\domain}(\zv) = +\infty $ then again nothing can be said about the vector $\uv$. Hence by convention it is assumed that $\xv \not \in \domain \ \Rightarrow\  \uv \in \emptyset $  
\end{enumerate}
By the above arguments we conclude that,
\begin{enumerate}
\item For $\xv \notin \setC$, $\partial \id_\setC(\xv) = \emptyset$
\item For $\xv \in \setC$, we have that $\wv \in  \partial \id_\setC(\xv)$ if $ \wv^\top(\zv-\xv) \le 0 \quad \forall \zv \in \setC$
\end{enumerate}

Hence the claim made in Lemma \ref{lem:subgrad_ind} is proved.
\end{proof}

\begin{proof}[\textbf{Proof of Lemma \ref{lem:general_constrained_lemma}}]
 \ From Lemma \ref{lem:subgrad_ind}, we know the expression for subgradient of the indication function $\id_{\setC}$
\begin{align}
 \partial g(\xv^\star)  =& \left\{ \sv  \ |\   \forall \zv \in \setC \; \ \sv^\top(\zv - \xv^\star) \le 0  \quad    \right\} \notag \\
  =& \left\{ \sv  \ |\   \forall \zv \in \setC  \; \ \sv^\top \zv \le \sv^\top \xv^\star   \quad     \right\}  \label{eq:subgrad_simplex}
\end{align}
Now, by the optimality condition \eqref{eq:opt_g}, $-A^\top \wv^\star \in \ \partial g(\xv^\star)$ and since this holds, hence $-A^\top \wv^\star$ should satisfy the required constrained which is needed to be in the set of subgradients of $\partial g(\xv^\star)$ according to conditions in equation \eqref{eq:subgrad_simplex}. Hence,
\begin{align}
 &\ (-A^\top \wv^\star)^\top \zv \le (-A^\top \wv^\star)^\top \xv^\star   \quad  \forall \zv \in \setC     \\
\Rightarrow & \ (A^\top \wv^\star)^\top \xv^\star  \le (A^\top \wv^\star)^\top \zv  \quad     \forall \zv \setC   \\
\Rightarrow  & \  (A^\top \wv^\star)^\top \xv^\star  \le \underset{z}{\min}(A^\top \wv^\star)^\top \zv  \quad     s.t \ \ \ \zv \in \setC  \\
\Rightarrow & \ (A \xv^\star )^\top \wv^\star  \leq  \underset{\zv \in \setC} {\min}(A\zv )^\top \wv^\star \quad     s.t \ \ \ \zv \in \setC 
\end{align}
Since $\xv^\star$ is a feasible point hence $(A \xv^\star )^\top \wv^\star  = \underset{\zv \in \setC} {\min}(A\zv )^\top \wv^\star \quad     s.t \ \ \ \xv^\star,\zv \in \setC$. 
\end{proof}

\subsection{Screening on Simplex Constrained Problems (Section \ref{subsec:simplex_screening})}
\label{app:simplex_screening}

\paragraph{General Simplex Constrained Screening}
\begin{proof}[\textbf{Proof of Theorem \ref{thm:simplexcondition}}]
In the simplex case, we have $g(\xv)=\id_\Simplex(\xv)$
and by Lemma \ref{lem:subgrad_ind} 
\begin{align}
 \partial g(\xv^\star)  =& \left\{ \sv  \ |\   \forall \zv \in \Simplex \; \ \sv^\top(\zv - \xv^\star) \le 0  \quad    \right\} \notag \\
  =& \left\{ \sv  \ |\   \forall \zv \in \Simplex \; \ \sv^\top \zv \le \sv^\top \xv^\star   \quad     \right\}  \label{eq:subgrad_simplex}
\end{align}
Now, by the optimality condition \eqref{eq:opt_g}, $-A^\top \wv^\star \in \ \partial g(\xv^\star)$ and since this holds, hence $-A^\top \wv^\star$ should satisfy the required constrained which is needed to be in the set of subgradients of $\partial g(\xv^\star)$ according to conditions in equation \eqref{eq:subgrad_simplex}. Hence,
\begin{align}
 &\ (-A^\top \wv^\star)^\top \zv \le (-A^\top \wv^\star)^\top \xv^\star   \quad  \forall \zv \in \Simplex     \\
\Rightarrow & \ (A^\top \wv^\star)^\top \xv^\star  \le (A^\top \wv^\star)^\top \zv  \quad     \forall \zv \in \Simplex  \\
\Rightarrow  & \  (A^\top \wv^\star)^\top \xv^\star  \le \underset{z}{\min}(A^\top \wv^\star)^\top \zv  \quad     s.t \ \ \ \zv \in \Simplex \ \\
\Rightarrow & \ (A^\top \wv^\star)^\top \xv^\star  \le \underset{i}{\min} \ \av_i^\top \wv^\star \ \label{min_explain_simplex}  \\
\Rightarrow & \ (A\xv^\star)^\top\wv^\star  \le \underset{i}{\min} \ \av_i^\top \wv^\star \  \\
\Rightarrow & \ (A\xv^\star)^\top\wv^\star  = \underset{i}{\min} \ \av_i^\top \wv^\star \  \label{last_line_simplex}
\end{align}

Equation \eqref{min_explain_simplex} is due to the fact that $\zv$ lie in the simplex, hence minimum value of $(A^\top \wv^\star)^\top \zv$ is $\underset{i}{\min} \ \av_i^\top \wv^\star$ and equation \eqref{last_line_simplex} also comes from the same fact that $\xv^\star$ lie in the simplex and hence $(A\xv^\star)^\top\wv^\star$ can not be smaller than $\underset{i}{\min} \ \av_i^\top \wv^\star$. That implies these two quantities need to be equal and all the $i$'s where this equality doesn't hold refers to $x_i^\star = 0$ for all such  $i$'s.
\begin{align*}
\ \av_i^\top \wv^\star > \ (A^\top \wv^\star)^\top \xv^\star  \ \Rightarrow\  x_i = 0 \\
\ (\av_i - A\xv^\star)^\top \wv^\star > 0 \ \Rightarrow\  x_i = 0
\end{align*}
\end{proof}

\begin{proof}[\textbf{Proof of Theorem \ref{thm:simplex_smooth}}]
From the optimality condition \eqref{eq:opt_f}, we have $\wv^\star = \nabla f(A\xv^\star)$ since $f$ is differentiable. Hence,
\begin{align}
(\av_i - A\xv^\star)^\top \wv^\star  &= (\av_i - A\xv^\star)^\top\nabla f(A\xv^\star) \qquad \qquad \qquad \qquad \qquad \qquad\\
&= (\av_i - A\xv^\star +  A\xv -  A\xv)^\top\nabla f(A\xv^\star) \\
&= (\av_i -  A\xv)^\top\nabla f(A\xv^\star) + (A\xv -A\xv^\star )^\top\nabla f(A\xv^\star) \\ 
&\ge (\av_i -  A\xv)^\top\nabla f(A\xv^\star) \quad \{\textrm{From the optimality of } f(A\xv)\} \label{eq:optmilatity_Ax} \\
&=(\av_i -  A\xv)^\top\nabla f(A\xv) - (\av_i -  A\xv)^\top(\nabla f(A\xv) - \nabla f(A\xv^\star))\\
&\ge (\av_i -  A\xv)^\top\nabla f(A\xv) - \norm{\av_i -  A\xv} \norm{\nabla f(A\xv) - \nabla f(A\xv^\star)}\\
&\ge (\av_i -  A\xv)^\top\nabla f(A\xv) - L\sqrt{\frac{\GW(\xv)}{\mu}}\norm{\av_i-A\xv} \label{lastline}
\end{align}
Eq. \eqref{eq:optmilatity_Ax} comes from the fact that at the optimal point $\xv^\star$, the inequality $(A\xv -A\xv^\star )^\top\nabla f(A\xv^\star) \ge 0$ holds $\forall \ \xv$. 
Equation \eqref{lastline} comes from Corollary \ref{cor:restric:smooth} for smooth function $f$ over a constrained set $\setC$.
\\
Hence from Theorem \ref{thm:simplexcondition}, we obtain the screening rule
$$(\av_i -A\xv)^\top\nabla f(A\xv) > L\sqrt{\frac{\GW(\xv)}{\mu}}\norm{\av_i-A\xv} \ \Rightarrow\  x_i^\star=0$$
\end{proof}

\paragraph{Screening for Squared Hinge Loss SVM.}
\begin{proof}[\textbf{Proof of Corollary \ref{cor:square_svm}}]
Theorem \ref{thm:simplex_smooth} is directly applicable to problems of the form \eqref{eq:sq:loss:svm}. The objective function $f(\yv) = f(A\xv)= \frac{1}{2}\xv^\top A^\top A\xv$ is strongly convex with parameter $\mu=1$. Also the derivative $\nabla f$ is Lipschitz-continuous with parameter $L=1$. 
To obtain an upper bound on the distance between any approximate solution and  the optimal solution $\|A\xv - A\xv^\star\|$, we employ Lemma \ref{lem:FWgaprestriction}. Since the constrained of the optimization problem is unit simplex and hence the value of Wolfe gap function $\GW(\xv) := \max_{\yv \in \domain}\; (A\xv-A\yv)^\top\nabla f(A\xv)$ as defined in Section \ref{sec:duality_gap_certificates} will be attained on one of the vertices. So, $\GW(\xv)=\max_{i \in 1\dots m}\; (A\xv-\av_i)^\top A\xv$. Finally, Theorem \ref{thm:simplex_smooth} gives us the screening rule for squared hinge loss SVM:
\begin{equation}
(\av_i-A\xv)^\top A\xv > \sqrt{\max_{i \in 1\dots m}\; (A\xv-\av_i)^\top A\xv}\norm{\av_i-A\xv} \ \Rightarrow\  x_i^\star = 0 
\end{equation}
\end{proof}

\paragraph{Screening on Minimum Enclosing Ball.}
\textbf{\\Minimum Enclosing Ball -} Given a set of $n$ points, $\av_1$ to $\av_n$ in $\R^d$, the minimum enclosing ball is defined as the smallest ball $B_{\cv,r}$ with center $\cv$ and radius $r$, i.e.: $B_{\cv,r} := \{\xv \in \R^d \ |\  \norm{\cv-\xv} \le r\}$, such that all points $\av_i$ lie in its interior. 
In this set-up, screening means to identify points $\av_i$ lying in the interior of the optimal ball $B_{\cv^\star,r^\star}$. Removing those points from the problem does not change the optimal ball.

\begin{proof}[\textbf{Proof of Corollary \ref{cor:min_enclosing_ball}}]
The minimum enclosing ball problem can be formulated as an optimization problem of the form given in Equation \eqref{eq:min_enclosing_ball}:
\begin{align*} 
& \underset{\cv,r}{\min} \ \ \ r^2  \quad \text{s.t.} \norm{\cv-\av_i}^2_2\le r^2 \quad \forall i \in [n]
\end{align*}
As we have seen, the dual formulation can be written in the form of Equation \eqref{eq:dualball} as given in \cite[Chapter 8.7]{Matousek:2007ub}:
\begin{align*}
 \underset{\xv}{\min} \ \ \xv^\top A^\top A\xv - \sum_{j=1}^{p} \av_j^\top\av_jx_j \quad \text{s.t. }  \xv \in \Simplex  
\end{align*}
Now the function $\xv^\top A^\top A\xv - \sum_{j=1}^{p} \av_j^\top\av_jx_j$ is strongly convex in $A\xv$ with parameter $\mu=2$. Since the constrained of the optimization problem is unit simplex and hence the value of the Wolfe gap function $\GW(\xv) := \underset{\yv \in \domain}{\max}\; (A\xv-A\yv)^\top\nabla f(A\xv)$ as defined in Section \ref{sec:duality_gap_certificates}  will be attained at one of the vertices of unit simplex. Hence Corollary \ref{cor:restric:smooth} gives $\GW(\xv)= \sqrt{\frac{1}{2}\max_{i}(\xv-\ev_i)^\top (2A^\top A\xv+\cv')}$. Now applying the findings of Theorem \ref{thm:simplex_smooth}, we get a sufficient condition for $\av_i$ to be non-influential, i.e. $\av_i$ lies in the interior of the MEB. But before that we will simplify the left hand side of the theorem \ref{thm:simplex_smooth} a bit. $(\av_i - A\xv)^\top \nabla f(A\xv)$ can we written as $(\bm{e}_i - \xv)^\top A^\top \nabla f(A\xv)$. Hence we get our result claimed in Corollary \ref{cor:min_enclosing_ball}.
\begin{equation}
\label{ballcor}
(\ev_i-\xv)^\top (2A^\top A\xv+\cv') > 2\sqrt{\tfrac{1}{2}\underset{j}{\text{max}}(\xv-\ev_i)^\top (2A^\top A\xv+\cv')}\norm{\av_i-A\xv} \ \Rightarrow\  x_i^\star = 0
\end{equation}
That means $\av_i$ is non influential.
\end{proof}

\subsection{Screening on $L_1$-ball Constrained Problems}\label{app:l1ball_screening}

\begin{proof}[\textbf{Proof of Theorem \ref{thm:l1_constrained_screening}}]
In the constrained Lasso case, we have $g(\xv)=\id_{\ball_{L_1}}(\xv)$ and by Lemma \ref{lem:subgrad_ind}
\begin{align}
 \partial g(\xv^\star)  =& \left\{ \sv  \ |\   \forall \zv \in {\ball_{L_1}} \; \ \sv^\top(\zv-\xv^\star) \le 0  \quad  \right\} \notag \\
  =& \left\{ \sv  \ |\   \forall \zv \in {\ball_{L_1}} \; \ \sv^\top \zv \le \sv^\top \xv^\star   \quad      \right\} \label{eq:subgrad_l1_ball}
\end{align}
Now, by the optimality condition \eqref{eq:opt_g}, $-A^\top \wv^\star \in \ \partial g(\xv^\star)$ and since this holds, hence $-A^\top \wv^\star$ should satisfy the required constrained which is needed to be in the set of subgradients of $\partial g(\xv^\star)$ according to conditions in equation \eqref{eq:subgrad_l1_ball}. Hence, 
\begin{align}
 &\ (-A^\top \wv^\star)^\top \zv \le (-A^\top \wv^\star)^\top \xv^\star   \quad  \forall \zv \in {\ball_{L_1}}     \\
\Rightarrow & \ (A^\top \wv^\star)^\top \xv^\star  \le (A^\top \wv^\star)^\top \zv  \quad     \forall \zv \in {\ball_{L_1}} \\
\Rightarrow  & \  (A^\top \wv^\star)^\top \xv^\star  \le \underset{z}{\min}(A^\top \wv^\star)^\top \zv  \quad     s.t \ \ \ \zv \in {\ball_{L_1}} \ \\
\Rightarrow & \ \ (A^\top \wv^\star)^\top \xv^\star  \le - \underset{i}{\max} \ \abs{\av_i^\top \wv^\star} \ \label{min_explain_l1}  \\
\Rightarrow & \ (A\xv^\star)^\top \wv^\star \le - \underset{i}{\max} \ \abs{\av_i^\top \wv^\star} \\
\Rightarrow & \ (A\xv^\star)^\top \wv^\star  = - \underset{i}{\max} \ \abs{\av_i^\top \wv^\star} \  \label{last_line_l1}
\end{align}
Equation \eqref{min_explain_l1} is due to the fact that $\zv$ lie in the $L_1$-ball and hence minimum value of $(A^\top \wv^\star)^\top \zv$ is $ - \underset{i}{\max} \ \abs{\av_i^\top \wv^\star}$ and Equation \eqref{last_line_l1} also comes from the same fact that $\xv^\star$ lie in the $L_1$-ball and hence $(A\xv^\star)^\top\wv^\star$ can not be smaller than $- \underset{i}{\max} \ \abs{\av_i^\top \wv^\star} $. That implies these two quantities need to be equal and all the $i$'s where this equality doesn't hold refers to $x_i^\star = 0$ for all such  $i$'s. Hence whenever these two quantities are not equal this holds:
\begin{align*}
-  \abs{\av_i^\top \wv^\star} > (A\xv^\star)^\top \wv^\star  \ \Rightarrow\  x_i^\star  = 0\\
\ \Rightarrow\  \abs{\av_i^\top \wv^\star}  + (A\xv^\star)^\top \wv^\star <0 \ \Rightarrow\  x_i^\star  = 0
\end{align*}
\end{proof}

\begin{proof}[\textbf{Proof of Theorem \ref{thm:l1_smooth_constrained}}] 
Using optimality condition \eqref{eq:opt_f}, we know that $ \wv^\star \in \partial f(A\xv)$
\begin{align} %
\abs{\av_i^\top \wv^\star}  + (A\xv^\star)^\top \wv^\star &= \abs{\av_i^\top \nabla f(A\xv^\star)}  + (A\xv^\star)^\top \nabla f(A\xv^\star) \quad\\ 
&= \abs{\av_i^\top ( \nabla f(A\xv) - \nabla f(A\xv) + \nabla f(A\xv^\star))}  + (A\xv^\star)^\top \nabla f(A\xv^\star)\\ 
&\le \abs{\av_i^\top \nabla f(A\xv)} + \abs{\av_i^\top ( \nabla f(A\xv^\star) - \nabla f(A\xv))} \notag\\
& \qquad \qquad \qquad +(A\xv^\star - A\xv +A\xv)^\top \nabla f(A\xv^\star)\\ %
&= \abs{\av_i^\top \nabla f(A\xv)} + \abs{\av_i^\top ( \nabla f(A\xv^\star) - \nabla f(A\xv))} \notag\\
& \qquad \qquad \qquad + (A\xv) ^\top \nabla f(A\xv^\star) - ( A\xv - A\xv^\star)^\top \nabla f(A\xv^\star) \label{eq:optimal_x_star_fAx}\\
&\le \abs{\av_i^\top \nabla f(A\xv)} + \abs{\av_i^\top ( \nabla f(A\xv^\star) - \nabla f(A\xv))}+(A\xv)^\top \nabla f(A\xv^\star) \label{eq_l1_proof_opt_f}\\ %
&\le \abs{\av_i^\top \nabla f(A\xv)} + \abs{\av_i^\top ( \nabla f(A\xv^\star) - \nabla f(A\xv))}\notag \\
& \qquad \qquad \qquad +(A\xv)^\top (\nabla f(A\xv^\star) - \nabla f(A\xv) +  \nabla f(A\xv))\\ 
&\le \abs{\av_i^\top \nabla f(A\xv)} + \abs{\av_i^\top ( \nabla f(A\xv^\star) - \nabla f(A\xv))} + (A\xv)^\top\nabla f(A\xv) \notag \\
& \qquad \qquad \qquad +(A\xv)^\top (\nabla f(A\xv^\star) - \nabla f(A\xv))\\ 
&\le \abs{\av_i^\top \nabla f(A\xv)} + (A\xv)^\top \nabla f(A\xv) + L(\norm{\av_i} + \norm{A\xv})\sqrt{\frac{\GW(\xv)}{\mu}}\label{eq_l1_proof_last} 
\end{align}
Eq. \eqref{eq:optimal_x_star_fAx} comes from the fact that at the optimal point $\xv^\star$, the inequality $(A\xv -A\xv^\star )^\top\nabla f(A\xv^\star) \ge 0$ holds $\forall \xv$.
Hence using Theorem \ref{thm:l1_constrained_screening}, Lemma \ref{lem:FWgaprestriction} and Corollary \ref{cor:restric:smooth}, we get the screening rule for $L_1$ constrained as whenever,\\ 
$\abs{\av_i^\top \nabla f(A\xv)} + (A\xv)^\top \nabla f(A\xv) + L(\norm{\av_i} + \norm{A\xv})\sqrt{\frac{\GW(\xv)}{\mu}} < 0 \ \Rightarrow\  \xv_i^\star = 0$
\end{proof}

\subsection{Screening on Elastic Net Constrained Problems}\label{app:elastic_net_screening}

\begin{proof}[\textbf{Proof of Theorem \ref{thm:elastic_net_constrained}}] 
Formulation :
\begin{align*}
&\min_{\xv} f(A\xv) \\
s.t \ \ & \alpha \|\xv\|_1 + \frac{(1 - \alpha) }{2} \|\xv\|_2^2 \leq 1 \\
& \Rightarrow \alpha \sum_{i=1}^n |\xv_i| + \frac{(1 - \alpha) }{2} \sum_{i = 1}^n \xv_i^2 \leq 1
\end{align*}
In the elastic net constrained case, we have $g(\xv)=\id_{\ball_{L_E}}(\xv)$ where $\id_{\ball_{L_E}}$ is elastic net norm ball. That implies 
$$\xv \in \id_{\ball_{L_E}} \ \ : \ \ \alpha \|\xv\|_1 + (1-\alpha) \|\xv\|_2^2 \leq 1$$
From the subgradient of indicator function and optimality condition for $A$ and $B$ framework
\begin{align}
 \partial g(\xv^\star)  =& \left\{ \sv  \ |\   \forall \zv \in {\ball_{L_1}} \; \ \sv^\top(\zv-\xv^\star) \le 0  \quad  \right\} \notag \\
  =& \left\{ \sv  \ |\   \forall \zv \in {\ball_{L_1}} \; \ \sv^\top \zv \le \sv^\top \xv^\star   \quad      \right\} \label{eq:subgrad_l1_ball}
\end{align}
Now, by the optimality condition \eqref{eq:opt_g}, $-A^\top \wv^\star \in \ \partial g(\xv^\star)$ and since this holds, hence $-A^\top \wv^\star$ should satisfy the required constrained which is needed to be in the set of subgradients of $\partial g(\xv^\star)$ according to conditions in equation \eqref{eq:subgrad_l1_ball}. Hence,

\begin{align}
 &\ (-A^\top \wv^\star)^\top \zv \le (-A^\top \wv^\star)^\top \xv^\star   \quad  \forall \zv \in {\ball_{L_E}}     \\
\Rightarrow & \ (A^\top \wv^\star)^\top \xv^\star  \le (A^\top \wv^\star)^\top \zv  \quad     \forall \zv \in {\ball_{L_E}} \\
\Rightarrow  & \  (A^\top \wv^\star)^\top \xv^\star  \le \underset{z}{\min}(A^\top \wv^\star)^\top \zv  \quad     s.t \ \ \ \zv \in {\ball_{L_E}} \
\end{align}

Since $\xv^\star$ is a feasible point hence $(A^\top \wv^\star)^\top \xv^\star  = \underset{z}{\min}(A^\top \wv^\star)^\top \zv  \quad     s.t \ \ \ \xv^\star,\zv \in {\ball_{L_E}}$. 
At the point where above equaliy hold $\xv^\star$ would be same as optimal $\zv$. Hence the problem reduces to,
\begin{align*}
&\min \ \  (A^\top \wv^\star)^\top \zv \\
s.t \ \ & \alpha \|\zv\|_1 + \frac{(1 - \alpha) }{2} \|\zv\|_2^2 \leq 1 \\
& \Rightarrow \alpha \sum_{i=1}^n |\zv_i| + \frac{(1 - \alpha) }{2} \sum_{i = 1}^n \zv_i^2 \leq 1
\end{align*}
Without the loss of generality let us assume that for $i \in \{1 \ldots m \}, \ \ z_i \geq 0$  and $i \in \{m+1 \ldots n \}, \ \ z_i \leq 0$. Hence the optimization problem can be written as :
\begin{align} \label{eq:min_elastic_ball}
&\min \ \  (A^\top \wv^\star)^\top \zv \\
s.t \ \ &  \alpha \big( \sum_{i=1}^m \zv_i  - \sum_{i=m+1}^n \zv_i \big ) + \frac{(1 - \alpha) }{2} \sum_{i = 1}^n \zv_i^2 \leq 1 \notag \\
& -z_i \leq 0  \ \ for \ \ i \in \{1 \ldots m \} \notag \\
& z_i \leq 0  \ \ for \ \ i \in \{m+1 \ldots n \}  \notag
\end{align}
Writing lagrangian for optimization problem \eqref{eq:min_elastic_ball}
\begin{align*}
\mathcal{L}(\zv,\lambda, u) =  (A^\top \wv^\star)^\top \zv  - \sum_{i=1}^m \lambda_i z_i + \sum_{i = m+1}^n \lambda_i z_i + u \Big (  \alpha \big( \sum_{i=1}^m \zv_i  - \sum_{i=m+1}^n \zv_i \big ) + \frac{(1 - \alpha) }{2} \sum_{i = 1}^n \zv_i^2 -1  \Big )
\end{align*}
Also optimization conditions are $\lambda_i \geq 0$, $\lambda_i z_i = 0$ and $\alpha \big( \sum_{i=1}^m \zv_i  - \sum_{i=m+1}^n \zv_i \big ) + (1 - \alpha) \sum_{i = 1}^n \zv_i^2  =1$. Also we conclude from above that if $\lambda_i > 0 \Rightarrow z_i = 0$. 
From first order optimality condition, \\
For $i \in \{1 \ldots m \}$ 
\begin{align} \label{eq:till_m1}
\av_i^\top \wv^\star - \lambda_i  =  - u\big( \alpha + (1-\alpha)|z_i|\big)
\end{align}
For $i \in \{m+1 \ldots n \}$ 
\begin{align} \label{eq:till_n1}
\av_i^\top\wv^\star + \lambda_i  = - u\big( \alpha + (1-\alpha)|z_i|\big)
\end{align}
Now in equations \eqref{eq:till_m1} and \eqref{eq:till_n1} we multiply by $z_i$ and add them. We get:
\begin{align} \label{eq:final_elastic}
(A^\top \wv^\star)^\top \zv + u [1 + \frac{(1 - \alpha) }{2} \|z\|_2^2] = 0
\end{align}
From equations \eqref{eq:till_m1}, \eqref{eq:till_n1}, \eqref{eq:final_elastic} and optimality conditions discussed above we get:
$$|\av_i^\top \wv^\star | + (A^\top \wv^\star)^\top z \Big [ \frac{ \alpha + (1 - \alpha)|z_i| }{ 1 + \frac{(1 - \alpha) }{2} \|z\|_2^2} \Big ] < 0  \Rightarrow z_i = 0$$
As discussed above $\xv^\star$ share same solution as optimal $z$. Hence
$$|\av_i^\top \wv^\star | + (A^\top \wv^\star)^\top \xv^\star \Big [ \frac{ \alpha  }{ 1 + \frac{(1 - \alpha) }{2} \|\xv^\star\|_2^2} \Big ] < 0  \Rightarrow x_i^\star = 0$$

\end{proof}

\begin{proof}[\textbf{Proof of Theorem \ref{thm:elastic_smooth_constrained}}] 
Using optimality condition \eqref{eq:opt_f}, we know that $ \wv^\star \in \partial f(A\xv)$
\begin{align} %
|\av_i^\top \wv^\star | + (A^\top \wv^\star)^\top \xv^\star & \Big [ \frac{ \alpha  }{ 1 + \frac{(1 - \alpha) }{2} \|\xv^\star\|_2^2} \Big ]  \leq |\av_i^\top \wv^\star | + (A^\top \wv^\star)^\top \xv^\star \big[\frac{2 \alpha}{3-\alpha} \big] \\
&= \abs{\av_i^\top \nabla f(A\xv^\star)}  + (A\xv^\star)^\top \nabla f(A\xv^\star) \big[\frac{2 \alpha}{3-\alpha} \big] \quad\\ 
&= \abs{\av_i^\top ( \nabla f(A\xv) - \nabla f(A\xv) + \nabla f(A\xv^\star))} \notag \\ 
& \qquad \qquad \qquad \qquad   \ \ + (A\xv^\star)^\top \nabla f(A\xv^\star) \big[\frac{2 \alpha}{3-\alpha} \big]\\ 
&\le \abs{\av_i^\top \nabla f(A\xv)} + \abs{\av_i^\top ( \nabla f(A\xv^\star) - \nabla f(A\xv))} \notag\\
& \qquad \qquad \qquad +(A\xv^\star - A\xv +A\xv)^\top \nabla f(A\xv^\star) \big[\frac{2 \alpha}{3-\alpha} \big] \\ %
&= \abs{\av_i^\top \nabla f(A\xv)} + \abs{\av_i^\top ( \nabla f(A\xv^\star) - \nabla f(A\xv))} \notag\\
&  + (A\xv) ^\top \nabla f(A\xv^\star) \big[\frac{2 \alpha}{3-\alpha} \big] - ( A\xv - A\xv^\star)^\top \nabla f(A\xv^\star) \big[\frac{2 \alpha}{3-\alpha} \big] \label{eq:optimal_x_star_fAx2}\\
&\le \abs{\av_i^\top \nabla f(A\xv)} + \abs{\av_i^\top ( \nabla f(A\xv^\star) - \nabla f(A\xv))} \notag \\ &\qquad \qquad  \ \ \qquad \qquad  \qquad +(A\xv)^\top \nabla f(A\xv^\star) \big[\frac{2 \alpha}{3-\alpha} \big] \label{eq_elastic_proof_opt_f}\\ %
&\le \abs{\av_i^\top \nabla f(A\xv)} + \abs{\av_i^\top ( \nabla f(A\xv^\star) - \nabla f(A\xv))}\notag \\
& \qquad \qquad  +(A\xv)^\top (\nabla f(A\xv^\star) - \nabla f(A\xv) +  \nabla f(A\xv))\big[\frac{2 \alpha}{3-\alpha} \big]\\ 
&\le \abs{\av_i^\top \nabla f(A\xv)} + \abs{\av_i^\top ( \nabla f(A\xv^\star) - \nabla f(A\xv))}  \notag \\
& + (A\xv)^\top\nabla f(A\xv)\big[\frac{2 \alpha}{3-\alpha} \big] +(A\xv)^\top (\nabla f(A\xv^\star) - \nabla f(A\xv))\big[\frac{2 \alpha}{3-\alpha} \big]\\ 
&\le \abs{\av_i^\top \nabla f(A\xv)} + (A\xv)^\top \nabla f(A\xv)\big[\frac{2 \alpha}{3-\alpha} \big] \notag \\
& \qquad \qquad \qquad \qquad+ L(\norm{\av_i} + \norm{A\xv} \big[\frac{2 \alpha}{3-\alpha} \big])\sqrt{\frac{\GW(\xv)}{\mu}}\label{eq_elastic_proof_last} 
\end{align}
Eq. \eqref{eq:optimal_x_star_fAx2} comes from the fact that at the optimal point $\xv^\star$, the inequality $(A\xv -A\xv^\star )^\top\nabla f(A\xv^\star) \ge 0$ holds $\forall \xv$.
Hence using Theorem \ref{thm:l1_constrained_screening}, Lemma \ref{lem:FWgaprestriction} and Corollary \ref{cor:restric:smooth}, we get the screening rule for $L_1$ constrained as whenever,\\ 
$$\abs{\av_i^\top \nabla f(A\xv)} + (A\xv)^\top \nabla f(A\xv) \big[ \frac{2 \alpha}{3-\alpha} \big ] + L(\norm{\av_i}_2 + \norm{A\xv}_2\big[ \frac{2 \alpha}{3-\alpha} \big ])\sqrt{\frac{\GW(\xv)}{\mu}} < 0  \notag
 \Rightarrow \ \xv_i^\star = 0$$
\end{proof}

\subsection{Screening for Box Constrained Problems}\label{app:box_screening}
\paragraph{Screening for General Box Constrained Problems (Section \ref{subsec:box_screening})}
\begin{proof}[\textbf{Proof of Theorem \ref{thm:box_constrained}}]
The box-constrained case can be seen in the form of the partially separable optimization problem pair \eqref{eq:primalS} and \eqref{eq:dualS}. According to optimality condition \eqref{eq:opt_gi} for this case, we have
\begin{align}
 -\av_i^\top \wv^\star \in&\ \partial g_i(x_i^\star) &&\forall i %
\end{align}
Now from the definition of subgradient for an indicator function as given in Lemma \ref{lem:subgrad_ind}.  Also since $x_i$ is a number now, we will get rid of the transpose here.
\begin{align}
 \partial g(x_i^\star)  =& \left\{ s  \ |\    0  \le z  \le C , \; \ s(z-x_i^\star) \le 0  \quad    \right\} \notag \\
  =& \left\{ s  \ |\    0  \le z \le C ,  \; \ s z \le s x_i^\star   \quad     \right\}  \label{eq:subgrad_box}
\end{align}

Now, by the optimality condition \eqref{eq:opt_gi}, $- \av_i^\top \wv^\star \in \ \partial g(x_i^\star)$ and since this holds, hence $- \av_i^\top \wv^\star$ should satisfy the required constrained which is needed to be in the set of subgradients of $\partial g(x_i^\star)$ according to conditions in Equation \eqref{eq:subgrad_box}. Hence,
\begin{align} \label{part_box}
&   \ (- \av_i^\top \wv^\star) z \le (- \av_i^\top \wv^\star) x_i^\star   \quad \forall z \ \textrm{ s.t }  \  0  \le z \le C , \;  \nonumber \\
\Rightarrow &\ \underset{z}{\min}(\av_i^\top \wv^\star) z \ge (\av_i^\top \wv^\star) x_i^\star   \quad s.t \ \    \ 0  \le z \le C   %
\end{align}
Now \eqref{part_box} can be manipulated in two ways \\
\begin{enumerate}
\item[Case 1] \begin{align*}
\av_i^\top \wv^\star > 0 \Rightarrow
& \ \underset{z}{\min}(\av_i^\top \wv^\star)^\top z \ge (\av_i^\top \wv^\star) x_i^\star   \quad  s.t \ \  \ 0  \le z \le C \\
\Rightarrow&   \ 0 \ge (\av_i^\top \wv^\star)^\top x_i^\star   \quad  
\end{align*}
But since $\av_i^\top \wv^\star  > 0$ and also $x_i^\star \ge 0$ hence $(\av_i^\top \wv^\star)x_i^\star \not < 0$. This implies $(\av_i^\top \wv^\star)x_i^\star = 0$ and hence if $ \av_i^\top \wv^\star > 0 \ \Rightarrow\  x_i^\star = 0$
\item[Case 2]
\begin{align*}
\av_i^\top \wv^\star < 0 \Rightarrow
& \ \underset{z}{\min}(\av_i^\top \wv^\star) z \ge (\av_i^\top \wv^\star) x_i^\star   \quad  s.t \ \  \ 0  \le z \le C \\
\Rightarrow&   \ (\av_i^\top \wv^\star)C \ge (\av_i^\top \wv^\star) x_i^\star   \quad  
\end{align*}
But since $\av_i^\top \wv^\star  < 0$ and also $x_i^\star \le C$ hence $(\av_i^\top \wv^\star)x_i^\star \not < (\av_i^\top \wv^\star)C$. This implies $(\av_i^\top \wv^\star)x_i^\star = (\av_i^\top \wv^\star)C $ and hence if $ \av_i^\top \wv^\star < 0 \ \Rightarrow\  x_i^\star = C$
\end{enumerate}
Final optimality arguments can be given as 
\begin{align} \label{eq:optimality_argument}
\av_i^\top \wv^\star > 0 \ \Rightarrow\  x_i^\star = 0 \notag \\
\av_i^\top \wv^\star < 0 \ \Rightarrow\  x_i^\star = C
\end{align}

Now 
\begin{align}
\av_i^\top \wv^\star = \av_i^\top (\wv^\star + \wv -\wv) = \av_i^\top \wv + \av_i^\top(\wv^\star - \wv) \notag \\
\av_i^\top \wv - \|\av_i\|_2 \|\wv - \wv^\star\|_2 \le \av_i^\top \wv^\star \le \av_i^\top \wv + \|\av_i\|_2\|\wv - \wv^\star\|_2 \label{eq:box_constr_util}
\end{align}
Since $f$ is $L$-Lipschitz gradient hence $f^*$ is $1/L$-strongly convex, hence using Lemmas \ref{lem:stronglyconvex} and \ref{lem:smooth_convex_condition}, Equation \eqref{eq:optimality_argument} becomes
\begin{align}
\av_i^\top \wv - \|\av_i\|_2 \sqrt{2L\,\G(\xv)} \le \av_i^\top \wv^\star \le \av_i^\top \wv + \|\av_i\|_2 \sqrt{2L\,\G(\xv)} \label{eq:box_constr_util2}
\end{align}
Hence using equation \eqref{eq:box_constr_util2} and earlier arguments we get,
\begin{align*}
&\av_i^\top \wv^\star > 0 \ \Rightarrow\  x_i^\star   =0 \nonumber \\
\Rightarrow &\av_i^\top \wv -  \norm{\av_i}_2\sqrt{2L\,\G(\xv)}>0 \ \Rightarrow\   x_i^\star =0 
\end{align*}
And if 
\begin{align*}
&\av_i^\top \wv^\star < 0 \ \Rightarrow\  x_i^\star   =C \nonumber  \\
\Rightarrow &\av_i^\top \wv +  \norm{\av_i}_2\sqrt{2L\,\G(\xv)}>0 \ \Rightarrow\   x_i^\star =C 
\end{align*}

\end{proof}

\paragraph{Screening on SVM with hinge loss and no bias}
\begin{proof}[\textbf{Proof of Corollary \ref{cor:SVMhinge}}]
Here the primal problem is given by:
\begin{equation}
\begin{aligned}
& \underset{\wv,\epsilonv}{\min} & & \tfrac{1}{2} \wv^\top\wv + C \one^\top\epsilonv\\
& \text{s.t.} & &  \wv^\top\av_i \ge 1-\epsilon_i & \forall i \in \{1:p\}\\
& & &  \epsilon_i \ge 0 & \forall i \in \{1:p\} 
\end{aligned}
\end{equation}
A dual formulation of the problem can be written as:
\begin{equation}
\begin{aligned}
& \underset{\xv}{\min} & & -\xv^\top\one + \tfrac{1}{2} \xv^\top A^\top A\xv\\
& \text{s.t.} & &  0 \le \xv \le C\one
\end{aligned}
\end{equation}
Theorem \ref{thm:box_constrained} is applied on the dual formulation. The objective function  $\frac{1}{2}\xv^\top A^\top A\xv - \xv^\top\one$  is  strongly convex with parameter $1$ and its derivative Lipschitz continuous with parameter $1$. The duality gap between primal and dual feasible points $\G(\wv,\epsilonv,\xv)$ is now used as suboptimality certificate which can play the role of the upper bound $\|\wv - \wv^\star\|$ using Lemma \ref{cor:Dgaprestriction}. For a given $\xv$ a primal feasible point can be obtained by setting $\wv=A\xv$ and $\epsilonv$ minimal such that the first constraint of the primal problem is satisfied. Using the obtained point for the duality gap, it only depends on the point $\xv$. All together this gives the screening rule:
\begin{eqnarray}
\av_i^\top A\xv +1 >   \norm{\av_i} \sqrt{2\G(\xv)} & \Rightarrow & x_i^\star = 0\\
\av_i^\top A\xv +1 < - \norm{\av_i} \sqrt{2\G(\xv)} & \Rightarrow & x_i^\star = C
\end{eqnarray}
\end{proof}
\textbf{Note -} Since the primal and dual of hinge loss SVM have very nice structure with smooth quadratic function with an addition to piece-wise linear convex function, hence it is not hard to show that both primal and dual function is $1$ strongly convex as shown in \cite{Zimmert:2015ue}. For more detailed proof, we recommend to go through \cite{Zimmert:2015ue}. Now  for an instance, if we write duality gap function as a function of $\wv$ then $$\G(\wv) \ge \G(\wv^\star) + \nabla \G(\wv^\star)^\top (\wv - \wv^\star) + \norm{\wv - \wv^\star}_2^2 $$ 
Since strong duality hold in SVM case, hence at optimal point $\wv^\star$,  $\G(\wv^\star) = 0$. Finally we get,
$$\G(\wv) \ge  \norm{\wv - \wv^\star}_2^2$$
Hence the screening rule comes out as given in \cite{Zimmert:2015ue}:
\begin{eqnarray}
\av_i^\top A\xv +1 >   \norm{\av_i} \sqrt{\G(\xv)} & \Rightarrow & x_i^\star = 0\\
\av_i^\top A\xv +1 < - \norm{\av_i} \sqrt{\G(\xv)} & \Rightarrow & x_i^\star = C
\end{eqnarray}

\section{Screening on Penalized Problems}
\label{app:penalized}

\subsection{Screening $L_1$-regularized Problems} \label{app:penalized_problem_l1}

\begin{lemma} \label{lem:l1_regularized}
Considering general $L_1$-regularized optimization problems
\begin{equation}\label{eq:l1_regularized}
\begin{aligned}
& \underset{\xv \in \R^p}{\min} & & f(A\xv) +  \lambda \norm{\xv}_1 
\end{aligned}
\end{equation}
At optimum points $\xv^\star$ and dual optimal point $\wv^\star$, the following rule is satisfied for the above problem formulation \eqref{eq:l1_regularized} :
\begin{align*}
\abs{\av_i^\top \wv^\star} < \lambda 
\ \Rightarrow \ x_i^\star = 0  %
\end{align*}
\end{lemma}

\begin{proof}[\textbf{Proof} ]

Since the optimization problem \eqref{eq:l1_regularized} comes under the partially separable framework and we can use the first order optimality condition \eqref{eq:opt_gi} as well as \eqref{eq:opt_gistar} to derive screening rules for the problem. Also we know that, the conjugate of the norm function is the indicator function of its dual norm ball. By the optimality condition \eqref{eq:opt_gistar}, we know that $$x_i \in \ \partial g_i^*(-\av_i^\top\wv)$$
here $g_i^*$ is the indicator function written as $\id_{L_\infty} (-\av_i^\top \wv )$. Hence for the indicator function $g^*$ by Lemma \ref{lem:subgrad_ind}
\begin{align*} 
 \partial g_i^*(-\av_i^\top\wv^\star)  =& \left\{ s \vert  \ \ \forall \zv \ s.t \ \ \abs{\frac{\av_i^\top \zv}{\lambda}} \le 1 \:; \ s(- \av_i ^\top \zv + \av_i^\top\wv^\star) \le 0  \quad    \right\} \\
  =& \left\{ s \vert  \ \ \forall \zv \ s.t \ \ \abs{\av_i^\top \zv }\le \lambda\:; \ s (\av_i ^\top \zv  )\ge s (\av_i^\top\wv^\star )\quad   \right\} 
\end{align*}
Since the optimality condition \eqref{eq:opt_gistar} holds hence $-x_i^\star$ should satisfy the required constrained which is needed to be in the set of subgradients of $\partial g_i^*(-\av_i^\top \wv^\star)$ according to conditions given above. That is 

\begin{align}
&-x_i^\star (\av_i ^\top \zv  )\le -x_i^\star (\av_i^\top\wv^\star )   \quad \forall \zv \ s.t \ \ \abs{\av_i^\top \zv }\le \lambda\  \\
&x_i^\star (\av_i ^\top \zv  )\ge x_i^\star (\av_i^\top\wv^\star ) \quad \forall \zv \ s.t \ \ \abs{\av_i^\top \zv }\le \lambda\  \\
\Rightarrow  & \  x_i^\star (\av_i^\top\wv^\star )  \le \underset{z}{\min}(x_i^\star (\av_i ^\top \zv  ))  \quad     s.t \ \ \ \abs{\av_i^\top \zv }\le \lambda\ \notag \\
\end{align}

\begin{itemize}
\item[  Case 1: ] $x_i^\star > 0$.
\begin{align}
&\  x_i^\star (\av_i^\top\wv^\star )  \le \underset{z}{\min}(x_i^\star (\av_i ^\top \zv  ))  \quad     s.t \ \ \ \abs{\av_i^\top \zv }\le \lambda\ \notag \\
& \Rightarrow \  x_i^\star (\av_i^\top\wv^\star )  \le -\lambda x_i^\star \notag \\
&\Rightarrow (\av_i^\top\wv^\star )  \le - \lambda  \notag \\
&\Rightarrow (\av_i^\top\wv^\star )  = - \lambda \label{eq:pen_dual_proof_part1}
\end{align}
Equation \eqref{eq:pen_dual_proof_part1} comes from the fact that $\abs{\av_i^\top \wv^\star} \le \lambda$
\item[Case 2:] $x_i ^\star< 0$.
\begin{align}
&\  x_i^\star (\av_i^\top\wv^\star )  \le \underset{z}{\min}(x_i^\star (\av_i ^\top \zv  ))  \quad     s.t \ \ \ \abs{\av_i^\top \zv }\le \lambda\ \notag \\
&\Rightarrow \  x_i^\star (\av_i^\top\wv^\star )  \le \lambda x_i^\star \notag \\
&\Rightarrow (\av_i^\top\wv^\star )  \ge  \lambda  \notag \\
&\Rightarrow (\av_i^\top\wv^\star )  =  \lambda \label{eq:pen_dual_proof_part2}
\end{align}
Equation \eqref{eq:pen_dual_proof_part1} comes from the fact that $\abs{\av_i^\top \wv^\star} \le \lambda$

\item[Case 3:] $x_i^\star = 0$.\\
Since if we assume $f$ as a continuous smooth function then $\av_i^\top \wv^\star$  is also continuous. Now if we consider arguments given for $x_i ^\star< 0$ and $x_i ^\star > 0$ we conclude that $\abs{\av_i^\top \wv^\star} = \lambda$ in all of the above two cases. Since $x_i^\star = 0$ is in the domain of the function \eqref{eq:primal}, hence at $x_i^\star = 0$, $\av_i^\top \wv^\star$ will lie in the open range of $-\lambda$ to $\lambda$. Which implies whenever $\abs{\av_i^\top \wv^\star} < \lambda$, then  $x_i^\star = 0$
\end{itemize}

Another view on the proof can be derived from the optimality condition \eqref{eq:opt_gi}.

The optimization problem \eqref{eq:l1_regularized} can be taken as partially separable problem and from the optimality condition \eqref{eq:opt_gi} kk
\begin{align}
&-\av_i^\top\wv^\star \in \partial g_i(x_i^\star) \label{eq:l1_1}\\ %
&\partial g_i(x_i^\star) \in  \left\{\begin{array}{ll}
		 \lambda \frac{x_i^\star}{\abs{x_i^\star}}   &\mbox{if } x_i \neq 0 \\ \relax
		[- \lambda,  \lambda ] & \mbox{if } x_i = 0
	\end{array} \right. \label{eq:l1_2}
\end{align}
From equations \eqref{eq:l1_1} and \eqref{eq:l1_2} we conclude that if
\begin{align*}
\abs{\av_i^\top \wv^\star} < \lambda \ \Rightarrow\  x_i^\star = 0
\end{align*}
\end{proof}

\begin{proof}[\textbf{Proof of Theorem \ref{thm:smooth_pen_lasso}}]
From Equation \eqref{eq:opt_f}, we know that $\wv^\star \in \partial f(A\xv^\star)$. Hence from Lemma \ref{lem:l1_regularized},
\begin{align}
\abs{\av_i^\top \wv^\star} &= \abs{\av_i^\top( \wv^\star -\wv  + \wv)} \notag \\
& \le \abs{\av_i^\top\wv} + \abs{\av_i^\top( \wv^\star - \wv)} \notag \\
& \le \abs{\av_i^\top\wv} + \norm{\av_i}_2\|\wv^\star - \wv\|_2 \notag \\
& \le \abs{\av_i^\top\wv} + \norm{\av_i}_2\sqrt{ 2L\,\G(\xv)} \label{eq:l1_reg_last} 
\end{align}
Eq. \eqref{eq:l1_reg_last} comes from Corollary \ref{cor:Dgaprestriction}. Now using Lemma \ref{lem:l1_regularized} and equation \eqref{eq:l1_reg_last}, we get
\begin{align*}
\abs{\av_i^\top \nabla f(A\xv)} < \lambda - \norm{\av_i}_2\sqrt{2L\,\G(\xv)} \ \Rightarrow\  \xv_i^\star=0
\end{align*}
\end{proof}

\paragraph{Penalized Lasso.}
Screening in this case can be derived from the existing ``gap safe'' paper \citep{Ndiaye:2015wj}. For completeness we here show that the same result follows from our Theorem \ref{thm:smooth_pen_lasso}:
\begin{corollary}\label{cor:pen_lasso}
\textbf{Penalized Lasso} Consider an optimization problem of the form:
\begin{align*}
\min_{\xv\in\R^n} \ \tfrac{1}{2} \|A\xv - b\|_2^2 + \lambda \|\xv\|_1
\end{align*}
Then the screening rule is given by:
$\abs{\av_i^\top(A\xv-\bv)} < \lambda - \norm{\av_i}_2\sqrt{2\G(\xv)} \ \Rightarrow\ \xv_i^\star=0.$
\end{corollary}
\begin{proof}[\textbf{Proof of Corollary \ref{cor:pen_lasso}}]
By observing the cost function for penalized lasso it can be concluded that \\
$$f(A\xv) = \tfrac{1}{2}\norm{A\xv-\bv}^2, \quad 
\wv=A\xv-\bv, \quad \textrm{and } L=1$$\\
Now results from Theorem \ref{thm:smooth_pen_lasso} can be directly applied here and hence the screening rule becomes \\
\begin{align*}
\abs{\av_i^\top(A\xv-\bv)} < \lambda - \norm{\av_i}\sqrt{2\G(\xv)} \ \Rightarrow\ \xv_i^\star=0.
\end{align*}
\end{proof}
This result is known in the literature \citep{Ndiaye:2015wj}, and we recover it using our proposed general approach in this paper by using Theorem \ref{thm:smooth_pen_lasso}.\\

Also, by applying same trick as mentioned after the end of proof of Corollary \ref{cor:SVMhinge}, we can show that we can get rid of the factor $2$ here also. Here also it is not hard to see that primal and dual (\eqref{eq:primal} and \eqref{eq:dual})  both are $1$ strongly convex in the dual variable $\wv$. Hence by the same argument as made in the proof of Corollary \ref{cor:SVMhinge}, we get that $$\G(\wv) \ge  \norm{\wv - \wv^\star}_2^2$$
And the improved screening rule comes out to be
$$\abs{\av_i^\top(A\xv-\bv)} < \lambda - \norm{\av_i}\sqrt{\G(\xv)} \ \Rightarrow\  \xv_i^\star=0.$$

\paragraph{Logistic Regression with $L_1$-regularization}
\begin{corollary} \label{cor:logistic_l1}
\textbf{Logistic Regression with $L_1$-norm Penalization.}
The optimization problem for logistic regression with $L_1$ regularizer can be written in the form of:
\begin{align} \label{eq:logistic_l1}
\min_{\xv\in\R^n} \ \sum_{i=1}^{n} \log(\exp([A\xv]_i)+1) + \lambda \|\xv\|_1
\end{align} 
And screening rule for above problem can be written as :
\begin{align*}
\abs{\av_i^\top\left(\frac{\exp(A\xv)}{\exp(A\xv)+1}\right)} < \lambda - \norm{\av_i}_2\sqrt{2\G(\xv)}
\ \Rightarrow \ \xv_i^\star=0
\end{align*}
where $\left(\frac{\exp(A\xv)}{\exp(A\xv)+1}\right)$ is element wise vector whose $i_{th}$ element is $\left(\frac{\exp([A\xv]_i)}{\exp([A\xv]_i)+1}\right)$
\end{corollary}

\begin{proof}[\textbf{Proof}]
By observation we know that in equation \eqref{eq:logistic_l1} $$f(A\xv) = \sum_{i=1}^{n} \log(\exp([A\xv]_i)+1) \text{ and }\wv \text{ is elementwise vector of } w_i \text{ s.t} \quad w_i=\frac{\exp([A\xv]_i)}{\exp([A\xv]_i)+1} $$
According to \cite[Lemma 5]{Smith:2015ua}, we get that the function $f(A\xv)$ is $1$-smooth. Hence $L = 1$ \\
Now from theorem \ref{thm:smooth_pen_lasso}, we derive the screening rule for logistic regression with $L_1$-regularization which is 
\begin{align*}
\abs{\av_i^\top\left(\frac{\exp(A\xv)}{\exp(A\xv)+1}\right)} < \lambda - \norm{\av_i}_2\sqrt{2\G(\xv)} \ \Rightarrow\  \xv_i^\star=0
\end{align*}
where $\left(\frac{\exp(A\xv)}{\exp(A\xv)+1}\right)$ is element wise vector whose $i_{th}$ element is $\left(\frac{\exp([A\xv]_i)}{\exp([A\xv]_i)+1}\right)$
This result is also known in the literature in \citep{Ndiaye:2015wj} (or see also \cite{Wang:2014un} for a similar approach) and we recover it using our prosed general approach in this paper by using Theorem \ref{thm:smooth_pen_lasso}.\\
\end{proof}

\paragraph{Elastic-net regularized regression}

\begin{proof}[\textbf{Proof of Corollary \ref{cor:elastic_net_regr}}]
\begin{align}
& \frac{1}{2}\|A\xv - \bv\|_2^2 + \lambda_2  \|\xv\|_2^2 + \lambda_1 \|\xv\|_1 \notag \\ 
&=  \frac{1}{2} [\xv^\top A^\top A\xv -2\bv^\top A\xv +\bv^\top\bv] + \lambda_2 \xv^\top\xv + \lambda_1 \|\xv\|_1 \notag\\
&=  \frac{1}{2}[\xv^\top(A^\top A + 2\lambda_2 I)\xv -2\bv^\top A\xv +\bv^\top\bv] + \lambda_1 \|\xv\|_1 \label{eq:elastic_formulate}
\end{align}
Now consider $A^\top A + 2\lambda_2 I = Q^\top Q$ and choose vector $\mv$ such that $A^\top\bv = Q^\top \mv$. Hence line \eqref{eq:elastic_formulate} can be written as
\begin{align*}
&  \tfrac{1}{2}\big[\xv^\top(A^\top A + 2\lambda_2 I)\xv -2\bv^\top A\xv +\bv^\top\bv \big]+ \lambda_1 \|\xv\|_1 \\
&=  \tfrac{1}{2}\big[\xv^\top Q^\top Q\xv - 2\mv^\top Q\xv + \mv^\top\mv - \mv^\top\mv + \bv^\top\bv \big] + \lambda_1 \|\xv\|_1 \\
&=  \tfrac{1}{2} \|Q\xv - \mv\|_2^2 + \tfrac{1}{2}\big[ \bv^\top\bv - \mv^\top\mv \big]  + \lambda_1 \|\xv\|_1
\end{align*}
Now the optimization problem \eqref{eq:elastic} can be written as 
\begin{equation} \label{eq:modified_elastic}
\min_{\xv} \ \tfrac{1}{2}\|Q\xv - \mv\|_2^2 +  \lambda_1 \|\xv\|_1
\end{equation}
Now results from Corollary \ref{cor:pen_lasso} can be directly applied to \eqref{eq:modified_elastic}.  \\
From observation, we know that $f(Q\xv) = \frac{1}{2}\norm{Q\xv-\mv}^2, \quad 
\wv=Q\xv-\mv, \quad \textrm{and }
L=1$\\
Simplification,
\begin{align}
\abs{\bm{q}_i^\top(Q\xv-\mv)} &= \abs{\bm{q}_i^\top Q\xv- \bm{q}_i^\top \mv}\notag\\
&= \abs{\bm{q}_i^\top Q\xv- \av_i^\top \bv}\notag \\
&= \abs{(\av_i^\top A + 2\lambda_2 \bm{e}_i^\top)\xv- \av_i^\top \bv} \label{eq:elastic_exp_smpl1}
\end{align}
\begin{align}
\abs{\bm{q}_i}\sqrt{2\G(\xv)} &= \sqrt{\av_i^\top \av_i + 2\lambda_2} \sqrt{2\G(\xv)}\notag \\
&=\sqrt{2(\av_i^\top \av_i +2 \lambda_2) \G(\xv)}\label{eq:elastic_exp_smpl2}
\end{align}
Now using results from Corollary \ref{cor:pen_lasso}, equations \eqref{eq:elastic_exp_smpl1} and \eqref{eq:elastic_exp_smpl2}, we get screening rules for elastic norm regularization regression problem as:
\begin{align*}
\abs{(\av_i^\top A + 2\lambda_2 \bm{e}_i^\top)\xv- \av_i^\top \bv} < \lambda_1 - \sqrt{2(\av_i^\top \av_i + 2\lambda_2)(\G(\xv))} \ \Rightarrow \ x_i^\star = 0.
\end{align*}
\end{proof}

\begin{lemma}[Conjugate  of  the  Elastic  Net  Regularizer [Lemma 6 \citep{Smith:2015ua}] ] \label{lem:smooth_elastic}
 For $\alpha \ \in  \ (0,1]$, the elastic net function $g(\xv) = \frac{1 - \alpha}{2} \|\xv\|_2^2 + \alpha \|\xv\|_1$ is the convex conjugate of $$g^*(\xv) = \sum_i \Big[ \frac{1}{2(1 - \alpha)}\big( \big[ |x_i|  - \alpha \big]_+\big)^2 \Big] = \sum_i  g_i^*(x_i)$$ where $g_i(\beta_i) =  \big[ \frac{1  - \alpha}{2} \beta_i^2 + \alpha |\beta_i| \big ]$ and   $[.]_+$ is the positive part operator, $[s]_+ = s \ \  for \ \ s > 0$
, and zero otherwise. Furthermore, this $g^*$ is is smooth, \textit{i.e.} has Lipschitz continuous gradient with constant $1/(1 - \alpha)$.

\end{lemma}

\begin{proof}[\textbf{Proof}]
The complete proof has been given in \citep[Lemma 6]{Smith:2015ua} but we also provide proof here below. \\
From the definition of convex conjugate function,
\begin{align}
g^*(\xv) &= \sup_{\betav}[\xv^\top \betav - g(\betav)] \notag \\
& = \sup_{\betav}\big[\xv^\top \betav - \big( \frac{1 - \alpha}{2} \|\betav\|_2^2 + \alpha \|\betav\|_1\big)\big] \notag \\
 & = \sup_{\beta_i}\big[\sum_{i} x_i \beta_i - \big( \sum_{i} \big( \frac{1 - \alpha}{2} \beta_i^2 + \alpha |\beta_i|\big)\big)\big] \ \ \forall \ \ i \in [n] \notag \\
 & = \sum_{i}  \sup_{\beta_i}  \big[x_i \beta_i -  \big( \frac{1  - \alpha}{2} \beta_i^2 + \alpha |\beta_i|\big) \big] \ \ \forall \ \ i \in [n] \notag  \\
 & = \sum_i g_i^*(x_i) \ \text{, where } \ \  g_i(\beta_i) =  \frac{1  - \alpha}{2} \beta_i^2 + \alpha |\beta_i| \notag
\end{align}
Now,
\begin{align*}
g_i^*(x_i) = \sup_{\beta_i}  \big[x_i \beta_i -  \big( \frac{1  - \alpha}{2} \beta_i^2 + \alpha |\beta_i|\big) \big]
\end{align*}
Consider three cases now : \\
\begin{itemize}
\item[Case 1:] $\beta > 0$.
 \begin{align*}
 g_i^*(x_i)   &= \sup_{\beta_i}  \big[x_i \beta_i -  \big( \frac{1  - \alpha}{2} \beta_i^2 + \alpha \beta_i \big) \big] \\
 &\Rightarrow \beta_i = \frac{(x_i - \alpha)}{(1 - \alpha)} \ \ \text{that also implies } \ \  x_i > \alpha \\
 \text{Hence, } g_i^*(x_i)   &= \frac{(x_i - \alpha)^2}{2(1 - \alpha)} \ \ \text{whenevr } \ \ x_i > \alpha
 \end{align*}
 \item[Case 2:] $\beta < 0$.
 \begin{align*}
 g_i^*(x_i)   &= \sup_{\beta_i}  \big[x_i \beta_i -  \big( \frac{1  - \alpha}{2} \beta_i^2 - \alpha \beta_i \big) \big] \\
 &\Rightarrow \beta_i = \frac{(x_i + \alpha)}{(1 - \alpha)} \ \ \text{that also implies } \ \  x_i < - \alpha \\
 \text{Hence, } g_i^*(x_i)   &= \frac{(x_i + \alpha)^2}{2(1 - \alpha)} \ \ \text{whenevr } \ \ x_i < -\alpha
 \end{align*}
 \item[Case 3:] $\beta = 0$.
 \begin{align*}
 g_i^*(x_i)   &= 0 \ \ \text{that also implies } \ \  |x_i| \leq  \alpha \\
 \end{align*}
\end{itemize} 

Hence, \begin{align*}
 g_i^*(x_i) = \frac{1}{2(1 - \alpha)}\big( \big[ |x_i|  - \alpha \big]_+\big)^2
\end{align*}
From all of the above arguments, $g^*(\xv) = \sum_i \Big[ \frac{1}{2(1 - \alpha)}\big( \big[ |x_i|  - \alpha \big]_+\big)^2 \Big] = \sum_i  g_i^*(x_i)$ 
\end{proof}

\begin{theorem} \label{thm:elasticnet_genral}
If we consider the general elastic net formulation of the form 
\begin{align}
\min_{\xv} \ f(A\xv) + \frac{(1-\alpha)}{2}  \|\xv\|_2^2 + \alpha \|\xv\|_1 \label{eq:elasticnet_genral}
\end{align}
If $f$ is L-smooth, then the following screening rule holds for all $i \in [n]$:
$$\abs{\av_i^\top \nabla f(A\xv)} < \alpha - \norm{\av_i}_2\sqrt{2L\,\G(\xv)} \ \Rightarrow\  \xv_i^\star=0$$ 
\end{theorem}

\begin{proof}[\textbf{Proof}] 
Since the optimization problem \eqref{eq:elasticnet_genral} comes under the partially separable framework and we can use the first order optimality condition \eqref{eq:opt_gi} as well as \eqref{eq:opt_gistar} to derive screening rules for the problem. 

 By optimality condition \eqref{eq:opt_gistar}, we know that $$x_i \in \ \partial g_i^*(-\av_i^\top\wv)$$
From lemma \ref{lem:smooth_elastic}, $g_i^*(-\av_i^\top \wv^\star) = \frac{1}{2(1 - \alpha)}\big( \big[ |\av_i^\top \wv^\star|  - \alpha \big]_+\big)^2$ and also  $ \partial g_i^*(-\av_i^\top\wv) = 0 \ \text{whenever} \ \abs{\av_i^\top\wv} \leq \alpha $ \\
Hence whenever $\abs{\av_i^\top\wv} \leq \alpha \Rightarrow x_i = 0$. \\

The same screening rule for elastic net regularized problem can be derived from the optimality condition  \eqref{eq:opt_gi}.
The optimization problem \eqref{eq:elasticnet_genral} can be taken as partially separable problem and from the optimality condition \eqref{eq:opt_gi} 
\begin{align}
&-\av_i^\top\wv^\star \in \partial g_i(x_i^\star) \label{eq:l1_1}\\ %
&\partial g_i(x_i^\star) \in  \left\{\begin{array}{ll}
		 \alpha \frac{x_i^\star}{\abs{x_i^\star}} + (1-\alpha) x_i   &\mbox{if } x_i \neq 0 \\ \relax
		[- \alpha,  \alpha ] & \mbox{if } x_i = 0
	\end{array} \right. \label{eq:l1_2}
\end{align}
Hence, whenever   $\abs{\av_i^\top\wv} \leq \alpha \Rightarrow x_i = 0$. \\

The above arguments also show the significance of symmetry in our formulation as structure \eqref{eq:primal} and \eqref{eq:dual}. This formulation provides our framework more flexibility to be used in larger class of problem.\\
Now,
\begin{align}
\abs{\av_i^\top \wv^\star} &= \abs{\av_i^\top( \wv^\star -\wv  + \wv)} \notag \\
& \le \abs{\av_i^\top\wv} + \abs{\av_i^\top( \wv^\star - \wv)} \notag \\
& \le \abs{\av_i^\top\wv} + \norm{\av_i}_2\|\wv^\star - \wv\|_2 \notag \\
& \le \abs{\av_i^\top\wv} + \norm{\av_i}_2\sqrt{ 2L\,\G(\xv)} \label{eq:elastic_reg_last} 
\end{align}
Equation \eqref{eq:elastic_reg_last} comes directly from corollary \ref{cor:Dgaprestriction}.  Hence finally we get the screening rules for general elastic net penalty problem which is very similar to screening for $L_1-penalized$ problems:

$$\abs{\av_i^\top \nabla f(A\xv)} < \alpha - \norm{\av_i}_2\sqrt{2L\,\G(\xv)} \ \Rightarrow\  \xv_i^\star=0$$

Now the above mentioned rule can be made a bit tighter under some condition which is not very interesting to discuss here.
\end{proof}

\subsection{Screening for Structured Norms} \label{app:structured_norm}

\begin{lemma} \label{lem:grplasso_basic}
If we use the same notation as mentioned in Section \ref{subsec:l1_l2_genral} to write a vector $\xv$ as a concatenation of smaller group vectors $\{ \xv_1 \cdots \xv_G\}$ such that $\xv^\top = \left[ \xv_1^\top, \xv_2^\top \cdots \xv_G^\top \right]$ and correspondingly the matrix $A$ can be denoted as the concatenation of column groups $A = [A_1 \ A_2 \cdots A_G]$. Now if we consider an optimization problem of the form $$\ \  \argmin_{\xv} f(A\xv) + \sum_{g=1}^G \sqrt{\rho_g} \|\xv_g\|_2 \qquad$$ \\
At the optimal point $\xv^\star$ and dual optimal points $\wv^\star$, we get rules according to the following equation: 
\begin{align*}
\|A_g^\top \wv^\star\|_2 < \sqrt{\rho_g} \ \Rightarrow\  \xv_g^\star = 0
\end{align*}        
\end{lemma}

\begin{proof}[\textbf{Proof}]

Dual of the problem is given by
\begin{align}
  \calP(\wv) 
    =f^*(\wv )
    +\sum_g \sqrt{\rho_g}  \id_{L_\infty} (\frac {\norm{A_g^\top \wv}_2}{\sqrt{\rho_g}} ) \label{eq:dual_l1_regularized}
\end{align}
Hence for the indicator function $g_g^*$ by Lemma \ref{lem:subgrad_ind} 
\begin{align*} 
 \partial g_g^*(-A_g^\top \wv^\star)  =& \left\{ \sv \vert  \ \ \forall \zv \ s.t \ \ \norm{\frac{A_g^\top  \zv}{\sqrt{\rho_g}}}_2 \le 1 \:; \ \sv^\top(-A_g^\top \zv+A_g^\top \wv^\star) \le 0  \quad    \right\} \\
  =& \left\{ \sv \vert  \ \ \forall \zv \ s.t \ \ \norm{A_g^\top  \zv}_2 \le \sqrt{\rho_g} ; \ \sv^\top (A_g^\top  \zv  )\ge \sv^\top  (A_g^\top \wv^\star )\quad   \right\} 
\end{align*}
Now, by the optimality condition \eqref{eq:opt_gistar} $\xv_g \in \ \partial g_g^*(-A_g^\top \wv^\star)$, and since this holds, hence $xv_g^\star$ should satisfy the required constrained which is needed to be in the set of subgradients of $\partial g^*(-A_g^\top \wv^\star)$ according to conditions given above. Hence,

\begin{align}
&-{\xv_g^\star}^\top (A_g^\top  \zv  )\le -{\xv_g^\star}^\top  (A_g^\top \wv^\star )  \quad \forall \zv \ s.t \ \ \norm{A_g^\top  \zv}_2 \le \sqrt{\rho_g}  \notag\\
&\Rightarrow {\xv_g^\star}^\top (A_g^\top  \zv  )\ge {\xv_g^\star}^\top  (A_g^\top \wv^\star )  \quad \forall \zv \ s.t \ \ \norm{A_g^\top  \zv}_2 \le \sqrt{\rho_g} \notag \\
&\Rightarrow   \  {\xv_g^\star}^\top  (A_g^\top \wv^\star ) \le \underset{z}{\min} \  {\xv_g^\star}^\top (A_g^\top  \zv  ) \quad   \ s.t   \ \ \norm{A_g^\top  \zv}_2 \le \sqrt{\rho_g} \notag \\
&\Rightarrow   \  {\xv_g^\star}^\top  (A_g^\top \wv^\star ) \le \underset{z}{\min} \|\xv_g\|_2 \|A_g^\top  \zv \|_2
\quad   \ s.t   \ \ \norm{A_g^\top  \zv}_2 \le \sqrt{\rho_g} \notag \\
&\Rightarrow   \  \xv_g^\top  (A_g^\top \wv^\star ) \le - \|\xv_g^\star\|_2 \sqrt{\rho_g} \notag \\
&\Rightarrow \|A_g^\top \wv^\star\|_2 = \sqrt{\rho_g}  \label{eq:aaaa}
\end{align}

Equation \eqref{eq:aaaa} comes from the cauchy inequality and true  $\forall \xv_g^\star \ : \ \ \xv_g^\star \neq 0$. Whenever $\|A_g^\top \wv^\star\|_2 < \sqrt{\rho_g}$ then $\xv_g^\star = 0$

Another view on the screening of above optimization problem can be seen from the optimality condition \eqref{eq:opt_gi}.
The optimization problem in Lemma \ref{lem:grplasso_basic} can be taken as partially separable problem and from the optimality condition \eqref{eq:opt_gi} %
\begin{align}
&-A_g^\top\wv^\star \in \partial g(\xv_g^\star) \label{eq:gl2l1_1}\\
&\partial g(\xv_g^\star) \in  \left\{\begin{array}{ll}
		 \sqrt{\rho_g} \frac{\xv_g}{\|x_g\|_2}   &\mbox{if } \xv_g \neq 0 \\ \relax
		\ball_2 & \mbox{if } \xv_g = 0 \ and \ \ball_2 \ \textrm{is norm ball of radius } \sqrt{\rho_g}
	\end{array} \right. \label{eq:gl2l1_2} 
\end{align}
From Equations \eqref{eq:gl2l1_1} and \eqref{eq:gl2l1_2}, we conclude that if 
$$\|A_g^\top \wv^\star\|_2 < \sqrt{\rho_g} \ \Rightarrow\ \xv_g^\star = 0$$
\end{proof}

\begin{proof}[\textbf{Proof of Theorem \ref{thm:smooth_grp_lasso}}]
From Equation \eqref{eq:opt_f}, we know that $\wv \in \nabla f(A\xv)$. Now
\begin{align}
\|  A_g^\top \wv^\star  \|_2 &= \| A_g^\top {(\wv + \wv^\star - \wv)}  \|_2  \leq \| A_g^\top \wv  \|_2 + \| A_g^\top (\wv^\star - \wv)\|_2 \notag \\ 
& = \|A_g ^\top \wv \|_2 + \sqrt{tr(( A_g^\top ( \wv^\star - \wv) )(( \wv^\star - \wv) ^\top) A_g)^\top} \notag \\
& \leq \|A_g^\top \wv  \|_2 + \sqrt{tr(( \wv^\star - \wv) ^\top( \wv^\star - \wv))} \sqrt{tr(A_g^\top A_g)} \notag \\
&=\|A_g^\top \wv \|_2 + \| \wv^\star -\wv \|_2 \frobnorm{A_g} \label{eq:grp_basic_last}
\end{align}
Using Corollary \ref{cor:Dgaprestriction}  with Equation \eqref{eq:grp_basic_last}, we get
\begin{align*}
\|  A_g^\top \wv^\star  \|_2 \le \|A_g^\top \nabla f(A\xv) \|_2 + \sqrt{2L\,\G(\xv)} \frobnorm{A_g}
\end{align*}
Hence using previous Lemma \ref{lem:grplasso_basic},
\begin{align*}
\|A_g^\top \nabla f(A\xv) \|_2 + \sqrt{2L\,\G(\xv)} \frobnorm{A_g} < \sqrt{\rho_g} \ \Rightarrow\  \xv_g^\star = 0
\end{align*}
\end{proof}

\begin{proof}[\textbf{Proof of Corollary \ref{grplassbasic_spc}}]
This is an explicit case of the optimization problem mentioned in Lemma \ref{lem:grplasso_basic}. By observation we know that,
$$f(A\xv) = \tfrac{1}{2}\|A\xv - b\|^2, \quad \wv = A\xv - b \quad and \quad L = 1$$
Now applying the findings of Theorem \ref{thm:smooth_grp_lasso}, we get 
\begin{align*}
\|A_g^\top (A\xv - \bv)\|_2 + \sqrt{2\G(\xv)}   \frobnorm{A_g} <\lambda \sqrt{\rho_g} \ \Rightarrow\  \xv_g^\star = 0
\end{align*}
\end{proof}

In Lemma \ref{lem:gengrpnorm} mentioned below, we show that the structured norm setting of \citep{Ndiaye:2015wj} can be derived from our more general \eqref{eq:primal} and \eqref{eq:dual} structure.
\begin{lemma} \label{lem:gengrpnorm}
Sparse Multi-Task and Multi Class Model \citep{Ndiaye:2015wj} - If we consider general problem of the form 
\begin{equation} \label{eq:gengrpnorm}
\min_{X \in \R^{p \times q}} \  \sum_{i = 1}^n f_i (\av_i^\top X) + \lambda\Omega(X) 
\end{equation}
where the regularization function $\Omega : \R^{p \times q}\rightarrow \R_+ $ is such that $\Omega(X) = \sum_{g=1}^p \| \xv_g \|_2$ and $X = [\xv_1, \xv_2 \cdots \xv_G]$. We write $W = [\wv_1, \wv_2 \cdots \wv_G]$ for variable of the dual problem. %
Then the screening rule becomes
\begin{align*}
\| {\av^{(g)}}^\top W \|_2 < \lambda  - \|\av^{(g)}\|_2 \, \|W - W^\star\|_2 \ \Rightarrow\  \xv_g^\star = 0
\end{align*}
Here $\av^{(g)}$ is the vector of the $g^{th}$ element group of each vector $\av_i$. 
\end{lemma}

\begin{proof}[\textbf{Proof}]
Equations pair \eqref{eq:primal} and \eqref{eq:dual} can be used interchangeably by replacing primal with dual and $f$ with $g$. Hence the partial separable primal-dual pair \eqref{eq:primalS} and \eqref{eq:dualS} can also be used interchangeably. By comparing Equation \eqref{eq:gengrpnorm} with \eqref{eq:primalS} and \eqref{eq:dualS}, we observe that separable function $\sum_{i = 1}^n f_i (\av_i^\top X)$ %
takes the place of separable $g^*$ in \eqref{eq:dualS}
and $\lambda\Omega(X)$ takes the place of $f^*$.
Hence we apply the optimality condition \eqref{eq:opt_fstar} to get (with exchanged primal dual variable) $$AW^\star \in \partial {\lambda \Omega(X^\star)}$$
Hence if,
\begin{align} \label{eq:primary_gen_grp_lass}
\| {\av^{(g)}}^\top W^\star \|_2 < \lambda \ \Rightarrow\  \xv_g = 0
\end{align}
Now,
\begin{align}
\| {\av^{(g)}}^\top W^\star \|_2 &= \| {\av^{(g)}}^\top (W^\star - W + W) \|_2 \notag\\
&\le \| {\av^{(g)}}^\top W \|_2 + \|{\av^{(g)}}^\top (W^\star - W ) \|_2 \notag\\
&\le \| {\av^{(g)}}^\top W \|_2 + \|{\av^{(g)}}\|_2 \|(W^\star - W ) \|_2 \label{eq:last_line_gen_grp_lass}
\end{align}
Using equations \eqref{eq:primary_gen_grp_lass} and \eqref{eq:last_line_gen_grp_lass}, the screening rule comes out to be
$$\| {\av^{(g)}}^\top W \|_2 < \lambda  - \|\av^{(g)}\|_2 \, \|W - W^\star\|_2 \ \Rightarrow\  \xv_g^\star = 0$$
\end{proof}
\begin{corollary}
If for all $i \in [n]$, $f_i$ is $L$-Lipschitz gradient then screening rule for equation \eqref{eq:gengrpnorm} is 
$$\| {\av^{(g)}}^\top W \|_2 < \lambda  - \|\av^{(g)}\|_2 \sqrt{2L\,\G(X)} \ \Rightarrow\  \xv_g^\star = 0$$ 
\end{corollary}
\begin{proof}
Using Lemma \ref{lem:gengrpnorm} and Corollary \ref{cor:Dgaprestriction}, we get the desired expression.
\end{proof}

\subsection{Connection with Sphere Test Method} \label{app:connection_sphere_test}
The general idea behind the sphere test method \cite{Xiang:2014vi}
is to consider the maximum value of desired function in a spherical region which contains the optimal dual variable. In context of our general framework \eqref{eq:primal} and \eqref{eq:dual}, we obtain this case when considering an $\ell_1$ penalty or $\ell_2/\ell_1$ penalty. That means $g$ is a norm and hence from Lemma \ref{lem:conjugates}, $g^*$ becomes the indicator function of the dual norm ball of $A^\top \wv$. The dual norm function for $\ell_1$ norm is of the form $\max_{i}|\av_i^\top \wv|$ and for $\ell_2/\ell_1$ norm, it is $\max_{g}\|A_g^\top \wv\|$. 
Hence, we try to find maximum value of the function of the forms  $\max_{\bm{\theta} \in \mathcal{S}(\bm{q},r)} \av_i^\top\bm{\theta}$ where $\mathcal{S}(\bm{q},r) = \{z : \|\zv-\bm{q}\|_2 \leq r \}$ the ball $\mathcal{S}$ also contains the optimal dual point $\wv^\star$. If the maximum value of $\av_i^\top\bm{\theta}$ is less than some particular value for all the $\bm{\theta}$ in the ball hence  $\av_i^\top \wv$ will also be less than that particular value and that is the main reason we try to find maximum of $\av_i^\top\bm{\theta}$ over the ball $\mathcal{S}$.
\begin{align*} 
\max_{\bm{\theta} \in \mathcal{S}(\bm{q},r)} \av_i^\top\bm{\theta} &= \av_i^\top(\bm{\theta} - \bm{q} + \bm{q}) =  \av_i^\top(\bm{\theta} - \bm{q}) + \av_i^\top \bm{q}\\
&\leq \|\av_i\|_2 \|\bm{\theta} - \bm{q}\| + \av_i^\top \bm{q} \leq r \|\av_i\|_2  + \av_i^\top \bm{q}
\end{align*}
Similar arguments can be given in the  $\ell_2/\ell_1$-norm case. A variety of existing screening test for lasso and group lasso are of this flavor of sphere tests. The difference between these approaches mainly lie in the way of choosing the center and bounding the radius of the sphere, such that the optimal dual variables lie inside the sphere.
Our method can be seen as a general framework for such a sphere test based screening with dynamic screening rules. Our method can be interpreted as a sphere test with the current iterate of the dual variable $\wv$ as a center of the ball, and we obtain the bound on the radius in terms of duality gap function.

\end{document}